\documentclass[letterpaper,11pt]{amsart}

\usepackage[margin=1.2in]{geometry}
\usepackage{amsmath,amsthm,amssymb}
\usepackage{xspace,xcolor}
\usepackage[breaklinks,colorlinks,citecolor=teal,linkcolor=teal,urlcolor=teal,pagebackref,hyperindex]{hyperref}
\usepackage[alphabetic]{amsrefs}
\usepackage{xypic}
\usepackage{tikz-cd}
\usepackage{comment}
\usepackage{xfrac}
\usepackage{mathrsfs}

\usepackage{color}

\theoremstyle{plain}
\newtheorem{thm}{Theorem}[section]

\newtheorem{lem}[thm]{Lemma}
\newtheorem{prop}[thm]{Proposition}
\newtheorem{cor}[thm]{Corollary}

\theoremstyle{definition}

\newtheorem{eg}[thm]{Example}

\theoremstyle{remark}
\newtheorem{rmk}[thm]{Remark}

\def\Z{{\mathbf Z}}
\def\Q{{\mathbf Q}}

\def\C{{\mathbf C}}

\def\A{{\mathbf A}}

\def\cC{\mathcal{C}}
\def\cD{\mathcal{D}}

\def\cF{\mathcal{F}}

\def\cH{\mathcal{H}}

\def\cJ{\mathcal{J}}

\def\cM{\mathcal{M}}
\def\cN{\mathcal{N}}
\def\cO{\mathcal{O}}

\def\cT{\mathcal{T}}

\def\bR{{\bf R}}

\def\a{\alpha}

\def\gr{{\rm Gr}}
\def\log{{\rm log}}
\def\dlog{{\rm dlog}}

\def\.{\cdot}
\def\^{\widehat}

\def\({\left(}
\def\){\right)}

\renewcommand{\and}{ \ \ \text{ and } \ \ }

\newcommand{\factor}[2]{\left. \raise 2pt\hbox{$#1$} \right/\hskip -2pt\raise -2pt\hbox{$#2$}}

\DeclareMathOperator{\codim} {codim}

\DeclareMathOperator{\lct} {lct}

\DeclareMathOperator{\id} {id}

\usepackage{soul}

\begin{document}

\author[Q.~Chen]{Qianyu Chen}

\address{Department of Mathematics, University of Michigan, 530 Church Street, Ann Arbor, MI 48109, USA}
\address{Institute of Geometry and Physics, University of Science and Technology of China, No.99, Xiupu Road, Pudong New Area, Shanghai, 200000, P.R.China}

\email{qyc@umich.edu, qianyu.chen16@gmail.com}

\author[M.~Musta\c{t}\u{a}]{Mircea Musta\c{t}\u{a}}

\address{Department of Mathematics, University of Michigan, 530 Church Street, Ann Arbor, MI 48109, USA}

\email{mmustata@umich.edu}

\title{A birational description of the minimal exponent}

\thanks{Q.C. was partially supported by an AMS-Simons Travel grant and M.M. was partially supported by NSF grant DMS-2301463 and by the Simons Collaboration grant \emph{Moduli of
Varieties}.}



\begin{abstract}
We give a description of the minimal exponent of a hypersurface using higher direct images of suitably twisted sheaves of log forms on a log resolution.
\end{abstract}

\maketitle

\section{Introduction}

Let $X$ be a smooth, irreducible, $n$-dimensional complex algebraic variety and let $Z$ be a hypersurface in $X$. The \emph{log canonical threshold} ${\rm lct}(X,Z)$ is an invariant of the singularities of $Z$
that plays a key role in birational geometry. This can be described in terms of a log resolution of $(X,Z)$, as follows. Suppose that $\pi\colon Y\to X$ is such a log resolution that we assume (for later purposes)
to be an isomorphism over $X\smallsetminus Z$. Therefore $\pi$ is proper, $Y$ is smooth, and $\pi^*(Z)=\sum_{i=1}^Na_iE_i$ is a simple normal crossing divisor. If we also write
$K_{Y/X}=\sum_{i=1}^Nk_iE_i$ (where $K_{Y/X}$ is the relative canonical divisor, locally defined by the determinant of the Jacobian matrix of $\pi$), then
\begin{equation}\label{formula_lct}
{\rm lct}(X,Z)=\min_{i=1}^N\tfrac{k_i+1}{a_i}.
\end{equation}
One can also interpret the log canonical threshold in terms of the \emph{multiplier ideals} of $Z$ (see \cite[Chapter~9.3.B]{Lazarsfeld}). Recall that for every $\lambda\in {\mathbf Q}_{\geq 0}$,
the multiplier ideal $\cJ(X,\lambda Z)$ is defined by
$$\cJ(X,\lambda Z)=\pi_*\cO_Y\big(K_{Y/X}-\lfloor\lambda\pi^*(Z)\rfloor\big).$$
With this definition, it is straightforward to see that
\begin{equation}\label{char_mult_ideals}
{\rm lct}(X,Z)=\min\big\{\lambda\in\Q_{\geq 0}\mid \cJ(X,\lambda Z)\neq\cO_X\big\}.
\end{equation}

Our goal in this note is to give an analogous description for a refinement of the log canonical threshold, the \emph{minimal exponent} $\widetilde{\alpha}(Z)$ (whenever the ambient variety 
is not clear from the context, we write $\widetilde{\alpha}(X,Z)$). This invariant was defined by Saito 
in \cite{Saito-B} using the Bernstein-Sato polynomial of a local equation of $Z$. By a result of Lichtin and Koll\'{a}r (see \cite[Section~10]{Kollar}), we have
$${\rm lct}(X,Z)=\min\big\{\widetilde{\alpha}(Z),1\big\}.$$
Furthermore, it was shown in \cite{Saito-B} that $\widetilde{\alpha}(Z)>1$ if and only if $Z$ has rational singularities. 
Recently, there has been interest in the minimal exponent due to the fact that it can be used to characterize higher Du Bois and higher rational singularities
(see \cite{MOPW}, \cite{Saito_et_al}, \cite{FL}, \cite{MP}).

Part of the motivation for our main results comes from the desire to have a birational description of the minimal exponent that is analogous to (\ref{formula_lct}). 
We recall that if $Z$ is reduced and the log resolution $\pi$ has the property that the strict transform of $Z$ is smooth, then it was shown in 
\cite[Corollary~D]{MPV} that we always have the inequality
\begin{equation}\label{formula_lct2}
\widetilde{\alpha}(Z)\geq\min_{i; E_i\,{\rm exc}}\frac{k_i+1}{a_i},
\end{equation}
where the minimum is over all \emph{exceptional} divisors $E_i$. 
Unfortunately, as pointed out by Koll\'{a}r, one can't hope to have equality in (\ref{formula_lct2}) since the right-hand side
depends on the log resolution. Therefore, describing the minimal exponent in terms of the numerical data of the resolution 
remains problematic. 
However, an outstanding open question in this direction is whether there is always an exceptional divisor $E_i$
such that $\widetilde{\alpha}(Z)=\frac{k_i+1}{a_i}$. 

The approach that we take in this paper goes in a slightly different direction: we aim for a cohomological characterization of the minimal exponent in terms of a log resolution
which extends the characterization (\ref{char_mult_ideals}) of the log canonical threshold via the triviality of multiplier ideals. In fact, 
the minimal exponent can be characterized in terms of the \emph{Hodge ideals} $I_p(\lambda Z)$, ideals introduced and studied in \cite{MPQ}. We note that for $p=0$, these recover the multiplier ideals:
we have $I_0(\lambda Z)=\cJ\big((\lambda-\epsilon)Z\big)$ for $0<\epsilon\ll 1$. It was shown in \cite{MPV} that for every nonnegative integer $p$ and every $\alpha\in (0,1]\cap\Q$, we have
$$\widetilde{\alpha}(Z)\geq p+\alpha\quad\text{if and only if}\quad I_p(\alpha Z)=\cO_X.$$
The Hodge ideals admit a description via a log resolution which extends the description of multiplier ideals, see \cite[Section~8]{MPQ}. However, this description is as the derived push-forward 
of a complicated complex on the log resolution and thus hard to relate to more concrete geometric information on the log resolution. Our goal in this note is to give a more explicit description of the minimal
exponent that only involves the higher direct images of twists of certain sheaves of log differentials on the resolution. 

Yet another description of the minimal exponent was given by Saito in \cite{Saito-MLCT} in terms of his \emph{microlocal multiplier ideals}, defined using the $V$-filtration with respect to a local equation of $Z$.
These ideals have been systematically studied recently in \cite{SY} under the name of \emph{higher multiplier ideals}. 
Our approach in this article is closer in spirit to this point of view since we use the description
of the minimal exponent in terms of the $V$-filtration.

\bigskip

Let us describe our main results now. Let $\pi\colon Y\to X$ be a log resolution as before and let $D=\pi^*(Z)$ and $E=D_{\rm red}$. Our goal is to describe the condition $\widetilde{\alpha}(Z)>\gamma$, for some
$\gamma\in\Q_{\geq 0}$,
using the higher direct images of the sheaves of differential forms with log poles along $E$, suitably shifted. If $Z$ is not reduced, then ${\rm lct}(X,Z)<1$ and thus $\widetilde{\alpha}(Z)={\rm lct}(X,Z)$. 
Therefore, from now on, we assume that $Z$ is reduced. 

We treat the cases $\gamma\in \Z$ and $\gamma\not\in\Z$ separately. 
If $\gamma=p\in \Z_{>0}$, then it was shown in \cite{MP} and \cite{FL} that the condition $\widetilde{\alpha}(Z)>p$ is equivalent to $Z$ having $(p-1)$-rational singularities. This is a condition that can be formulated
in terms of a resolution of singularities of $Z$. The following result rephrases this condition in terms of the log resolution of $(X,Z)$. Let us denote by $Z_{\rm sing}$ the singular locus of $Z$.

\begin{thm}\label{thm_char}
If $Z$ is a reduced hypersurface in $X$ and $p\in \Z_{>0}$, then $\widetilde{\alpha}(Z)>p$ if and only if 
the following two conditions hold:
\item[i)] We have
\begin{equation}\label{eq_thm_char}
R^q\pi_*\Omega_Y^i({\rm log}\,E)=0\quad\text{for all}\quad q\geq 1, i\leq p.
\end{equation}
\item[ii)] We have ${\rm codim}_Z(Z_{\rm sing})\geq 2p$.
\end{thm}
In fact, the proof will show that in ii) above, it is enough to assume that ${\rm codim}_Z(Z_{\rm sing})\geq \min\{3,2p\}$ (see Remark~\ref{codim_sing_locus}). The proof of Theorem~\ref{thm_char} 
follows easily from the results in \cite{MP}.

Our main contribution in this paper is to treat the case when $\gamma\not\in\Z$. In this case, we write $\gamma=p+\alpha$, with $p\in\Z$ and $\alpha\in (0,1)$. 

\begin{thm}\label{thm_main}
Let $Z$ be a reduced hypersurface in $X$. For every $\alpha\in (0,1)\cap\Q$ and $p\in \Z_{\geq 0}$, the following hold:
\begin{enumerate}
\item[i)] If $\widetilde{\alpha}(Z)\geq p$, then
\begin{equation}\label{eq_thm_main}
R^q\pi_*\Omega_Y^{n-p}({\rm log}\,E)\big(-E-\lfloor \alpha D\rfloor\big)\to R^q\pi_*\Omega_Y^{n-p}{(\rm log}\,E)(-E)
\end{equation}
is an isomorphism for all $q\neq p$ and it is injective for $q=p$.
\item[ii)] If $\widetilde{\alpha}(Z)>p$, then we have
$\widetilde{\alpha}(Z)>p+\alpha$ if and only if the morphism (\ref{eq_thm_main}) is an isomorphism for $q=p$.
\end{enumerate}
\end{thm}
Note that if $p=0$, then $\Omega_Y^{n-p}({\rm log}\,E)(-E)=\omega_Y$, so the assertion in ii) says that ${\rm lct}(X,Z)>\alpha$ if and only if $\cJ(X,\alpha Z)=\cO_X$. 

Rewriting the criterion in Theorem~\ref{thm_main} via Grothendieck duality and using the vanishing in Theorem~\ref{thm_char} gives the following reformulation:

\begin{cor}\label{dual_formulation}
Let $Z$ be a reduced hypersurface in $X$. If $\alpha\in (0,1)\cap\Q$ and $p\in \Z_{\geq 0}$ is such that $\widetilde{\alpha}(Z)>p$, then $\widetilde{\alpha}(Z)>p+\alpha$ if and only if
$$R^q\pi_*\Omega_Y^p({\rm log}\,E)\big(\lfloor \alpha D\rfloor\big)=0\quad\text{for all}\quad q\geq 1.$$
\end{cor}

Note that if in Corollary~\ref{dual_formulation} we take $\alpha=1-\epsilon$, for $0<\epsilon\ll 1$, then we recover the fact that if $\widetilde{\alpha}(Z)>p$, then
we have $\widetilde{\alpha}(Z)\geq p+1$ if and only if the canonical morphism
$\gamma_p\colon \Omega_Z^p\to\underline{\Omega}_Z^p$ is an isomorphism, where $\underline{\Omega}_Z^p$ is the $p$th Du Bois complex of $Z$, see 
Remark~\ref{rmk_DuBois} below. In fact, it is known that $\widetilde{\alpha}(Z)\geq p+1$ if and only if $\gamma_i$ is an isomorphism for all $i\leq p$ (that is,
$Z$ has $p$-\emph{Du Bois singularities}), see \cite{MOPW} and \cite{Saito_et_al}. 

While Corollary~\ref{dual_formulation} can be used to compute the minimal exponent in some simple cases (most notably, 
for ordinary singularities), it is not well-suited for computations (even the case of ordinary singularities requires considerable work). 
We view this characterization of the minimal exponent as having rather theoretical interest. 
For example, as an application of Corollary~\ref{dual_formulation}, we prove 
the constancy of the minimal exponent in a proper family of hypersurfaces that admit a simultaneous log resolution
(see Theorem~\ref{thm_last}). This answers a question of Radu Laza.

\medskip

The key step in the proof of Theorem~\ref{thm_main} is to describe the condition for having $\widetilde{\alpha}(Z)>p+\alpha$, when $\alpha\in (0,1)\cap\Q$ and we know that
$\widetilde{\alpha}(Z)\geq p+\alpha$. We now outline the main ideas in the proof. Since the assertion is local, we may assume that $Z$ is defined by some $f\in\cO_X(X)$. 
In this case, using the $V$-filtration associated to $f$, one defines the nearby cycles 
$$\psi_f(\cO_X)=\bigoplus_{\alpha\in (0,1]\cap\Q}\psi_{f,\alpha}(\cO_X).$$
This underlies a mixed Hodge module in the sense of Saito's theory \cite{Saito-MHM}.
In particular, it is a left module\footnote{We follow the convention that the underlying $\cD$-modules of Hodge modules are left $\cD$-modules.} over the sheaf of differential operators $\cD_X$ on $X$, endowed with a good filtration, the Hodge filtration $F_{\bullet}\psi_f(\cO_X)$,
which is the direct sum of the Hodge filtrations $F_{\bullet}\psi_{f,\alpha}(\cO_X)$. 
Therefore, the de Rham complex ${\rm DR}_X\psi_{f,\alpha}(\cO_X)$ becomes a filtered complex and for every $q$, the graded piece ${\rm Gr}^F_q{\rm DR}_X\psi_{f,\alpha}(\cO_X)$
is a complex of coherent $\cO_X$-modules. 

Saito's description of $\widetilde{\alpha}(Z)$ in terms of the $V$-filtration implies that if $\alpha\in (0,1)\cap\Q$ and $\widetilde{\alpha}(Z)\geq p+\alpha$, then we have 
$\widetilde{\alpha}(Z)>p+\alpha$ if and only if $\cH^0\big({\rm Gr}^F_{p+1-n}{\rm DR}_X\psi_{f,\alpha}(\cO_X)\big)=0$. On the other hand, given our log resolution $\pi$,
we have
$$\psi_{f,\alpha}(\cO_X)\simeq\pi_*\psi_{g,\alpha}(\cO_Y),$$
where $g=f\circ\pi$ and the push-forward is in the derived category of mixed Hodge modules. Saito's Strictness theorem thus implies that
$${\rm Gr}^F_{p+1-n}{\rm DR}_X\psi_{f,\alpha}(\cO_X)\simeq {\mathbf R}\pi_*{\rm Gr}^F_{p+1-n}{\rm DR}_Y\psi_{g,\alpha}(\cO_Y).$$
The main point is to give an explicit description of ${\rm Gr}^F_{p+1-n}{\rm DR}_Y\psi_{g,\alpha}(\cO_Y)$: the advantage is that in this case the divisor defined by $g$
has simple normal crossings. We do this by giving an explicit filtered resolution of $\psi_{g,\alpha}(\cO_Y)$ for all $\alpha\in (0,1]$ in Theorem~\ref{thm_resolution}: this is our main technical result. We note that a related result appears in \cite[Proposition~15.12.5]{MHM_project} (we are grateful to Ruijie Yang for pointing this out).

The paper is structured as follows. In Section~\ref{section2}, after a brief review concerning $\cD$-modules and Hodge modules, we recall the $V$-filtration, the nearby cycles, and the 
description of the minimal exponent that we need. The following section contains the proof of Theorem~\ref{thm_char}. We describe the filtered resolution of $\psi_{g,\alpha}(\cO_Y)$
in Section~\ref{section_resolution} and give the proofs of Theorem~\ref{thm_main} and Corollary~\ref{dual_formulation} in Section~\ref{section_final}. The final section
contains the application to proper families of hypersurfaces having simultaneous log resolutions.

\subsection{Acknowledgment} We are grateful to Brad Dirks, Mihnea Popa, Christian Schnell, and Ruijie Yang for many useful discussions. We would like to thank also the anonymous referee for a detailed list of suggestions to improve the presentation.

\section{Background concerning Hodge modules, $V$-filtrations, and the minimal exponent}\label{section2}

Let $X$ be a smooth, irreducible complex algebraic variety and let $n=\dim(X)$. 
We denote by $\cD_X$ the sheaf of differential operators on $X$. 
We begin by recalling some basic facts about $\cD$-modules; for details, we refer the reader to \cite{HTT}. Unless explicitly mentioned otherwise, by a $\cD$-module
we mean a \emph{left} $\cD_X$-module. By an algebraic system of coordinates on an open subset $U\subseteq X$, we mean $x_1,\ldots,x_n\in\cO_X(U)$ such that
$dx_1,\ldots,dx_n$ trivialize $\Omega_U$. We denote the corresponding basis of ${\mathcal T}_U={\mathcal Der}_{\C}(\cO_U)$ by $\partial_{x_1},\ldots,\partial_{x_n}$.

For a $\cD_X$-module $\cM$, its \emph{de Rham complex} ${\rm DR}_X\cM$ is the complex
$$0\to\cM\overset{\nabla}\longrightarrow \Omega_X^1\otimes_{\cO_X}\cM\overset{\nabla}\longrightarrow\ldots\overset{\nabla}\longrightarrow\Omega_X^n\otimes_{\cO_X}\cM\to 0,$$
placed in cohomological degrees $-n,\ldots,0$. 
Recall that a \emph{good filtration} on a coherent $\cD_X$-module $\cM$ is an increasing filtration $F_{\bullet}\cM$ by coherent $\cO_X$-modules,
compatible with the order filtration $F_{\bullet}\cD_X$ on $\cD_X$, such that ${\rm Gr}^F_{\bullet}(\cM)$ is locally finitely generated over ${\rm Gr}^F_{\bullet}(\cD_X)\simeq {\mathcal Sym}(\cT_X)$.
Given such a filtration, the de Rham complex of $\cM$ becomes a filtered complex such that its graded piece ${\rm Gr}^F_{p-n}{\rm DR}_X\cM$ is the complex of coherent
$\cO_X$-modules
$$0\to {\rm Gr}^F_{p-n}(\cM)\to \Omega_X^1\otimes_{\cO_X}{\rm Gr}^F_{p-n+1}(\cM)\to\ldots\to \Omega_X^n\otimes_{\cO_X}{\rm Gr}^F_{p}(\cM)\to 0.$$

Recall that there is an equivalence of categories between left and right $\cD_X$-modules, such that the right $\cD_X$-module $\cM^r$ corresponding to a left $\cD_X$-module $\cM$
has underlying $\cO_X$-module $\omega_X\otimes_{\cO_X}\cM$. One first checks that $\omega_X=\Omega_X^n$ has a canonical right $\cD_X$-module structure described in a chart with algebraic coordinates $x_1,\ldots,x_n$
by 
$$(hdx_1\wedge\ldots\wedge dx_n)\cdot\partial_{x_i}=-\tfrac{\partial h}{\partial x_i}dx_1\wedge\ldots\wedge dx_n.$$
The right action of $D\in {\rm Der}_{\C}(\cO_X)$ on $\omega_X\otimes_{\cO_X}\cM$ is then given by
\begin{equation}\label{eq_description_right_action}
(\eta\otimes u)\cdot D=\eta D\otimes u-\eta\otimes Du\quad\text{for all}\quad \eta\in\omega_X, u\in\cM.
\end{equation}
For future reference, we mention the following identity, which can be easily checked, concerning the $\cD_X$-action on $\omega_X$: if $x_1,\ldots,x_n$ are local algebraic coordinates on $X$, then
\begin{equation}\label{eq_identity_omega}
\sum_{i=1}^n(dx_i\wedge\eta)\partial_{x_i}=-d\eta\quad\text{for all}\quad\eta\in\Omega_X^{n-1}.
\end{equation}
We also note that if  $F_{\bullet}\cM$ is a good filtration on the left coherent $\cD_X$-module $\cM$, then we get a corresponding good filtration on $\cM^r$, with the indexing convention given by
$$F_{p-n}\cM^r=\omega_X\otimes_{\cO_X}F_p\cM\quad\text{for all}\quad p\in\Z.$$

\medskip

While we will not be using Saito's theory of Hodge modules explicitly, some of the objects we will consider and some of the results we will make use of have their proper context in this theory,
for which we refer to \cite{Saito-MHP} and \cite{Saito-MHM}. We only mention here that a mixed Hodge module consists of several pieces of data, the most relevant for us being a $\cD_X$-module
with a fixed good filtration, the \emph{Hodge filtration}. A key example is ${\mathbf Q}_X[n]$, whose underlying $\cD_X$-module is $\cO_X$, with the Hodge filtration given by
$F_p\cO_X=\cO_X$ for $p\geq 0$ and $F_p\cO_X=0$ for $p<0$. 

All morphisms of mixed Hodge modules are strict\footnote{This means that if $u\colon (\cM,F_{\bullet}\cM)\to (\cN,F_{\bullet}\cN)$ is such a morphism,
then $u(\cM)\cap F_p\cN=u(F_p\cM)$ for all $p\in\Z$.} with respect to the Hodge filtration and the category of mixed Hodge modules on $X$ is an abelian category. The corresponding bounded derived category is denoted 
$D^b\big({\rm MHM}(X)\big)$ and these derived categories satisfy a 6-functor formalism. By taking the graded pieces of the filtered de Rham complex, we obtain exact functors
$${\rm Gr}^F_{p-n}{\rm DR}_X\colon D^b\big({\rm MHM}(X)\big)\to D^b_{\rm coh}(X),$$
for all $p\in\Z$, where on the right-hand side we have the bounded derived category of $\cO_X$-modules, with coherent cohomology. 

\medskip

We next recall the notion of $V$-filtration, due to Malgrange and Kashiwara, which will play an important role in what follows. 
We only discuss the case of the $\cD_X$-module $\cO_X$:
for a discussion of the general case, for details, and for the role of the $V$-filtration in the theory of Hodge modules, we refer to \cite[Section~3.1]{Saito-MHP} (see also \cite{Schnell} for a nice introduction).

Suppose that $f\in\cO_X(X)$ is a nonzero regular function defining the hypersurface $Z$. We consider the graph embedding
 $\iota\colon X\hookrightarrow X\times {\mathbf A}^1$, given by $\iota(x)=\big(x,f(x)\big)$, and let
 $B_f:=\iota_+(\cO_X)$.
 We have\footnote{This is a consequence of the well-known fact that if $i\colon X\hookrightarrow Y$ is a closed immersion of smooth varieties, of pure relative dimension $r$, then there is an isomorphism of left $\cD_Y$-modules $i_+(\cO_X)\simeq\cH^r_{i(X)}(\cO_Y)$, where on the right-hand side we have the $r$th local cohomology sheaf of $\cO_Y$ along $i(X)$. This isomorphism can be checked using the description of $i_+(\cO_X)$ in \cite[Example~1.3.5]{HTT} and the computation of local cohomology via the \v{C}ech complex.}
 $$B_f=\cO_X[t]_{f-t}/\cO_X[t]
 =\bigoplus_{j\geq 0}\cO_X\partial_t^j\delta_f,$$
 where $\delta_f$ denotes the class of $\tfrac{1}{f-t}$. 
For computations, it is useful to keep in mind the following rules concerning the action of $\cD_{X\times\A^1}$:
\begin{equation}\label{eq1_action_on_B_f}
D\cdot h\partial_t^j\delta=D(h)\partial_t^j\delta-hD(f)\partial_t^{j+1}\delta\quad\text{for all}\quad D\in {\rm Der}_{\C}(\cO_X), h\in\cO_X, j\in\Z_{\geq 0},
\end{equation} 
\begin{equation}\label{eq2_action_on_B_f}
t\cdot h\partial_t^j\delta=fh\partial_t^j\delta-jh\partial_t^{j-1}\delta\quad\text{for all}\quad h\in\cO_X, j\in\Z_{\geq 0}.
\end{equation}
 
 The $V$-filtration is a decreasing, exhaustive filtration $(V^{\alpha}B_f)_{\alpha\in\Q}$ on $B_f$ by quasi-coherent $\cD_X$-modules,
 which is discrete and left-continuous\footnote{This means that there is a positive integer $\ell$ such that $V^{\alpha}B_f$ is constant for $\alpha\in \big(\tfrac{i-1}{\ell},\tfrac{i}{\ell}]$ for all $i\in\Z$.},
 and uniquely characterized by a few properties. We only mention two of these:
 \begin{enumerate}
 \item[i)] $t\cdot V^{\alpha}B_f\subseteq V^{\alpha+1}B_f$ and $\partial_t\cdot V^{\alpha}B_f\subseteq V^{\alpha-1}B_f$ for all $\alpha\in\Q$;
 \item[ii)] $(\partial_tt-\alpha)$ is nilpotent on ${\rm Gr}_V^{\alpha}B_f$ for all $\alpha\in\Q$.
 \end{enumerate}
 Here, we put ${\rm Gr}_V^{\alpha}B_f=V^{\alpha}B_f/V^{>\alpha}B_f$, where $V^{>\alpha}B_f=V^{\alpha+\epsilon}B_f$, for $0<\epsilon\ll 1$. 
  
 The Hodge filtration on $B_f$ is given by 
 $$F_{p+1}B_f=\bigoplus_{0\leq j\leq p}\cO_X\partial_t^j\delta_f\quad\text{for all}\quad p\in\Z.$$
 For every $\alpha\in\Q$, we get an induced Hodge filtration on ${\rm Gr}_V^{\alpha}B_f$.

 If $\alpha\in (0,1]\cap\Q$, then the {nearby cycles} of $\cO_X$ with respect to $f$, corresponding to $\alpha$, are given by 
 $\psi_{f,\alpha}(\cO_X):={\rm Gr}_V^{\alpha}B_f$ and 
 $$\psi_f(\cO_X):=\bigoplus_{\alpha\in (0,1]}\psi_{f,\alpha}(\cO_X),$$
 with the Hodge filtration given by the direct sum of the filtrations on each $\psi_{f,\alpha}(\cO_X)$.
 It is an important fact that $\psi_f(\cO_X)$, with the Hodge filtration, is the filtered $\cD_X$-module
 underlying a mixed Hodge module on $X$.

 When working with right $\cD$-modules (in which case $\cO_X$ is replaced by the right $\cD_X$-module $\omega_X$), the $V$-filtration on $B_f$ induces a $V$-filtration
 on $B_f^r=\iota_+(\omega_X)=\omega_X\otimes_{\cO_X}B_f$. We write the elements of $B_f^r$ as $\sum_{i\geq 0}\eta_i\partial_t^i\delta$, with $\eta_i\in\omega_X$.
 In this setting, it is customary to index the $V$-filtration increasingly, according to the convention
 $$V_{\alpha}B_f^r=\omega_X\otimes_{\cO_X}V^{-\alpha}B_f\quad\text{for all}\quad\alpha\in\Q.$$
 We also note that, according to our conventions about filtrations and the fact that we have $\dim(X\times\A^1)=n+1$, the
 Hodge filtration on $B_f^r$ is given by
 $$F_{p-n}B_f^r=\bigoplus_{0\leq i\leq p}(\omega_X\otimes_{\cO_X}\cO_X\partial_t^i\delta).$$
 Since $\partial_t$ acts on the right on $B_f^r$ via the action of $-\partial_t$ on $B_f$ on the left, it follows from condition ii) in the definition of the $V$-filtration
 that the action of $\theta-\alpha$ on ${\rm Gr}_{\alpha}^V(B_f^r)$ is nilpotent, where $\theta$ is the Euler operator $t\partial_t$. 

 \begin{rmk}\label{right_nearby_cycles}
 We note that one has to be careful when switching to right $\cD$-modules in the description of nearby cycles: the right $\cD_X$-module $\psi_{f,-\alpha}(\omega_X):=\psi_{f,\alpha}(\cO_X)^r$
 is canonically isomorphic to ${\rm Gr}^V_{-\alpha}\iota_+(\omega_X)$, which carries an induced Hodge filtration $F_{\bullet}$ from $\iota_+(\omega_X)$. 
 However, since we do the switch left-right on an $n$-dimensional variety (while for $B_f$ this was done on an $(n+1)$-dimensional variety), the Hodge filtration on $\psi_{f,-\alpha}(\omega_X)$ is given by $F_{\bullet}[-1]$, that is
 $$F_p\psi_{f,-\alpha}(\omega_X)=F_{p-1}{\rm Gr}^V_{-\alpha}\iota_+(\omega_X)\quad\text{for all}\quad p\in\Z.$$
 \end{rmk}
 
 \begin{rmk}
 As we have already mentioned, one can define $V$-filtrations with respect to $f$ for more general $\cD_X$-modules ${\mathcal M}$ (for example, for all $\cD_X$-modules
 underlying Hodge modules on $X$). This is a filtration on $\iota_+(\cM)$ that satisfies analogous properties to those we discussed in the case $\cM=\cO_X$. 
 For details, see \cite{Saito-MHP}. We will only need this briefly in the proof of Proposition~\ref{prop_birat_nearby_cycles} below.
 \end{rmk}
 
  For future reference, we state here a well-known fact about the behavior of (graded de Rham complexes of) nearby cycles 
 under birational morphisms.
 
 \begin{prop}\label{prop_birat_nearby_cycles}
 Let $X$ be a smooth, irreducible, complex algebraic variety and $f\in\cO_X(X)$ nonzero. If $\pi\colon Y\to X$ is a proper morphism
 that is an isomorphism over the complement of the hypersurface defined by $f$, with $Y$ smooth, and $g=f\circ\pi$, then 
 \begin{equation}\label{eq2_funct_nearby}
{\rm Gr}^F_i{\rm DR}_X\psi_{f,\alpha}(\cO_X)\simeq {\mathbf R}\pi_*\big({\rm Gr}^F_i{\rm DR}_Y\psi_{g,\alpha}(\cO_Y)\big)
\end{equation}
for all $i\in\Z$ and all $\alpha\in (0,1]\cap\Q$.
 \end{prop}
 
 \begin{proof}
 Indeed,
 on the one hand, if we consider $\pi\times\id\colon Y\times {\mathbf A}^1\to X\times {\mathbf A}^1$, then
 we have isomorphisms of filtered $\cD$-modules
$${\rm Gr}_V^{\alpha}{\mathcal H}^j\big((\pi\times\id)_+B_g\big)\simeq \cH^j(\pi_+{\rm Gr}_V^{\alpha}B_g)\quad\text{for all}\quad j\in\Z$$
(see \cite[Proposition~3.3.17]{Saito-MHP}).
Note that here, on the left-hand side, we have the $V$-filtration with respect to a $\cD$-module which is not the structure sheaf.
 On the other hand, since $\pi$ is an isomorphism over the complement of the hypersurface $Z$ defined by $f$, we have
$\cH^0\big((\pi\times \id)_+(B_g)\big)\simeq B_f\oplus\cM$, where $\cM$ is a filtered $\cD_{X\times \A^1}$-module supported on $X\times \{0\}$, while
$\cH^j\big((\pi\times \id)_+(B_g)\big)$ is supported on $X\times \{0\}$ for all $j\neq 0$. Since $\alpha>0$, we have ${\rm Gr}_V^{\alpha}(\cN)=0$ for every 
$\cD_{X\times \A^1}$-module $\cN$ supported on $X\times \{0\}$ by \cite[Lemme~3.1.3]{Saito-MHP}. We thus conclude that
\begin{equation}\label{eq1_funct_nearby}
\psi_{f,\alpha}(\cO_X)\simeq\pi_+\big(\psi_{g,\alpha}(\cO_Y)\big).
\end{equation}
Using the compatibility of the graded de Rham complex with push-forward for (direct summands of) Hodge modules, which is a consequence
of Saito's Strictness theorem (see \cite[Th\'{e}or\`{e}me~5.3.1]{Saito-MHP}), we obtain the formula in the proposition.
 \end{proof}
 
 \medskip
 
 We next recall the definition of Saito's minimal exponent, see \cite{Saito-B}. 
 Let $Z$ be a nonempty hypersurface in $X$ and suppose first that $Z$ is defined by some $f\in\cO_X(X)$.
 Recall that the Bernstein-Sato polynomial of $f$ is the monic polynomial $b_f(s)$ of minimal degree such that
 $$b_f(s)f^s\in\cD_X[s]\cdot f^{s+1}.$$
 It is easy to check that since $f$ is not invertible, we have $b_f(-1)=0$. The \emph{minimal exponent}
 $\widetilde{\alpha}(f)$ is the negative of the largest root of $b_f(s)/(s+1)$ (with the convention that this is $\infty$ if
 $b_f(s)=s+1$). By a result of Kashiwara \cite{Kashiwara}, all roots of $b_f(s)$ are in $\Q_{<0}$. In particular, we have
 $\widetilde{\alpha}(f)\in\Q_{>0}\cup\{\infty\}$. 
 
 For an arbitrary hypersurface $Z$, we choose a finite open cover $X=\bigcup_{i\in I}U_i$ such that $Z\cap U_i$ is defined in $U_i$ by some 
 $f_i\in\cO_X(U_i)$. In this case, the minimal exponent of $Z$ is given by
 $$\widetilde{\alpha}(Z)=\min_i\widetilde{\alpha}(f_i),$$
 where the minimum is over the noninvertible $f_i$.
 One can check that this is independent of the choice of cover and defining equations. By a result of Lichtin and Koll\'{a}r
 (see \cite[Section~10]{Kollar}), we have ${\rm lct}(X,Z)=\min\big\{\widetilde{\alpha}(Z),1\big\}$. 
 
 \begin{rmk}\label{rmk_local_min_exponent}
 Actually, the definition in \cite{Saito-B} concerns the \emph{local minimal exponent} $\widetilde{\alpha}_P(Z)$ of $Z$ at a point $P\in Z$. The definition is the same as above,
 but using the \emph{local} Bernstein-Sato polynomial $b_{f,P}(s)$ of a local equation $f$ of $Z$ at $P$. It is well-known that if $Z$ is defined by $f$, then
 $b_f(s)$ is the least common multiple of the $b_{f,P}(s)$, for all $P\in Z$. Moreover, given any point $P\in Z$, there is an open neighborhood $U$ of $P$
 such that $b_{f,P}(s)=b_{f\vert_U}(s)$. This implies that $\widetilde{\alpha}(Z)=\min_{P\in Z}\widetilde{\alpha}_P(Z)$ and for every $P\in Z$,
 there is an open neighborhood $U$ of $P$ such that $\widetilde{\alpha}_P(Z)=\widetilde{\alpha}(Z\cap U)$.
 \end{rmk}
 
 We will make use of the description of $\widetilde{\alpha}(Z)$ in terms of the $V$-filtration, which we now recall.
 Suppose that $Z$ is defined by the nonzero $f\in\cO_X(X)$. The following result is due to Saito:
 
 \begin{thm}\label{char_Saito}
 If $\gamma\in\Q_{>0}$ and we write
 $\gamma=p+\alpha$, where $p\in {\mathbf Z}_{\geq 0}$ and $\alpha\in (0,1]\cap {\mathbf Q}$, then the following are equivalent:
 \begin{enumerate}
 \item[i)] $\widetilde{\alpha}(f)\geq p+\alpha$.
 \item[ii)] Around every point in $X$, there is a section of $V^{\alpha}B_f$ of the form
 $\partial_t^p\delta+\sum_{i=1}^ph_i\partial_t^{p-i}\delta$.
 \item[iii)] $\partial_t^p\delta\in V^{\alpha}B_f$.
 \item[iv)] $F_{p+1}B_f\subseteq V^{\alpha}B_f$. 
 \end{enumerate}
 \end{thm}
 
 \begin{proof}
 For the equivalence between i) and ii), see \cite[(1.3.8)]{Saito-MLCT} (where the condition in ii) is formulated in terms of the microlocal $V$-filtration). 
 The implications iv)$\Rightarrow$iii)$\Rightarrow$ii) are trivial, hence it is enough to prove, by induction on $p$, the implication ii)$\Rightarrow$iv). 
 The case $p=0$ is clear, hence we may assume $p\geq 1$. If $u=\partial_t^p\delta+\sum_{i=1}^ph_i\partial_t^{p-i}\delta\in V^{\alpha}B_f$, then 
 $$-\tfrac{1}{p}(t-f)u=\partial_t^{p-1}\delta+\sum_{i=1}^{p-1}g_i\partial_t^{p-1-i}\delta\in V^{\alpha}B_f,$$
 for suitable $g_1,\ldots,g_{p-1}\in\cO_X$. By induction, we get $F_{p}B_f\subseteq V^{\alpha}B_f$ and together with the hypothesis in ii), we conclude that $F_{p+1}B_f\subseteq V^{\alpha}B_f$. 
 \end{proof}
 
 We will make use of the following two immediate consequences of the above theorem:
 
 \begin{cor}\label{cor1_char}
 If $\alpha\in (0,1)\cap \Q$ and $p\in\Z_{\geq 0}$,
 then the following hold:
 \begin{enumerate}
 \item[i)] If $\widetilde{\alpha}(f)\geq p$, then ${\rm Gr}^F_{i}\psi_{f,\alpha}(\cO_X)=0$ for all $i\leq p$.
 \item[ii)] If $\widetilde{\alpha}(f)\geq p+\alpha$, then 
 $\widetilde{\alpha}(f)>p+\alpha$ if and only if ${\rm Gr}^F_{p+1}\psi_{f,\alpha}(\cO_X)=0$. 
 \end{enumerate}
 \end{cor}
 
 \begin{proof}
 Let us put $F_qV^{\beta}B_f:=F_qB_f\cap V^{\beta}B_f$. It follows from the definition of the induced filtration
 that for every $i$, we have
 $${\rm Gr}^F_{i}\psi_{f,\alpha}(\cO_X)=F_iV^{\alpha}B_f/(F_{i-1}V^{\alpha}B_f+F_iV^{>\alpha}B_f).$$
 The assertion in i) thus follows from the fact that since $\widetilde{\alpha}(f)\geq p$, by Theorem~\ref{char_Saito} we have
 $$F_pB_f\subseteq V^1B_f\subseteq V^{>\alpha}B_f.$$
 
 Suppose now that $\widetilde{\alpha}(f)\geq p+\alpha$. Using the theorem again, we see that $F_{p+1}B_f\subseteq V^{\alpha}B_f$. Since we also have
 $F_pB_f\subseteq V^{>\alpha}B_f$, we conclude that
 $${\rm Gr}^F_{p+1}\psi_{f,\alpha}(\cO_X)\simeq \cO_X/J,$$
 where $J=\{h\in\cO_X\mid h\partial_t^p\delta\in V^{>\alpha}B_f\}$. The assertion in ii) follows by one more application of the theorem.
 \end{proof}
 
 It is convenient to also give a version of the above corollary in terms of the graded de Rham complex of the nearby cycles:
 
 \begin{cor}\label{cor2_char}
  If $\alpha\in (0,1)\cap \Q$ and $p\in\Z_{\geq 0}$ is such that
 $\widetilde{\alpha}(f)\geq p$, then the following hold:
 \begin{enumerate}
 \item[i)] ${\rm Gr}^F_{i-n}{\rm DR}_X\psi_{f,\alpha}(\cO_X)=0\quad\text{for}\quad i\leq p$.
 \item[ii)] $\cH^q\big({\rm Gr}^F_{p+1-n}{\rm DR}_X\psi_{f,\alpha}(\cO_X)\big)=0\quad\text{for}\quad q\neq 0$.
 \item[iii)] $\cH^0\big({\rm Gr}^F_{p+1-n}{\rm DR}_X\psi_{f,\alpha}(\cO_X)\big)\simeq \omega_X\otimes_{\cO_X}{\rm Gr}^F_{p+1}\psi_{f,\alpha}(\cO_X)$.
 \end{enumerate}
 In particular, if $\widetilde{\alpha}(f)\geq p+\alpha$, then we have $\widetilde{\alpha}(Z)>p+\alpha$ if and only if
$$\cH^0\big({\rm Gr}^F_{p+1-n}{\rm DR}_X\psi_{f,\alpha}(\cO_X)\big)=0.$$
 \end{cor}
 
 \begin{proof}
 It follows from the definition of the graded de Rham complex that the complex ${\rm Gr}^F_{i-n}{\rm DR}_X\psi_{f,\alpha}(\cO_X)$ involves
 ${\rm Gr}^F_j{\rm Gr}_V^{\alpha}B_f$, with $i-n\leq j\leq i$. In particular, it follows from assertion i) in the previous corollary that since 
 $\widetilde{\alpha}(f)\geq p$, all terms in the complex ${\rm Gr}^F_{i-n}{\rm DR}_X\psi_{f,\alpha}(\cO_X)$ are $0$ if $i\leq p$.
 Similarly, all terms in the complex ${\rm Gr}^F_{p+1-n}{\rm DR}_X\psi_{f,\alpha}(\cO_X)$ that lie in nonzero degrees are 0, while
 the term in degree $0$ is $\omega_X\otimes_{\cO_X}{\rm Gr}^F_{p+1}\psi_{f,\alpha}(\cO_X)$. This gives the assertions in i), ii), and iii).
 The last assertion in the corollary follows from iii) and the assertion ii) in the previous corollary.
 \end{proof}

\section{The condition $\widetilde{\alpha}(Z)>p$ for $p\in {\mathbf Z}_{>0}$}

In this section, we describe the condition for having $\widetilde{\alpha}(Z)>p$, where $p$ is a positive integer, in terms of a log resolution. As in the previous sections,
we consider a smooth, irreducible, $n$-dimensional complex algebraic variety, and a nonempty hypersurface $Z$ in $X$. Let $\pi\colon Y\to X$
be a log resolution of $(X,Z)$ that is an isomorphism over $X\smallsetminus Z$ and let $E=\pi^*(Z)_{\rm red}$.
Our goal is to prove Theorem~\ref{thm_char}.

\begin{rmk}\label{codim_sing_locus}
We will see in the proof of Theorem~\ref{thm_char} given below that in ii) we may only require that ${\rm codim}_Z(Z_{\rm sing})\geq\min\{3,2p\}$.
\end{rmk}

\begin{proof}[Proof of Theorem~\ref{thm_char}]
As we will see, 
this proof follows easily from the results in \cite{MP}. Let $r={\rm codim}_Z(Z_{\rm sing})$. 
Note first that the vanishing condition in (\ref{eq_thm_char}) is independent of the choice of log resolution: see \cite[Corollary~31.2]{MP3}.
Therefore we may assume that  $\pi$ is a composition of blow-ups with smooth centers lying over $Z_{\rm sing}$ and let $\pi_Z\colon \widetilde{Z}\to Z$
be the restriction of $\pi$ to the strict transform of $Z$. We write $E=\widetilde{Z}+F$, where $F$ is the sum of the $\pi$-exceptional divisors. With this notation,
it follows from \cite[Proposition~1.1]{MP} that 
\begin{equation}\label{eq2_thm_char}
R^q\pi_*\Omega_Y^{i+1}({\rm log}\,E)\simeq R^q{\pi_Z}_*\Omega_{\widetilde{Z}}^i({\rm log}\,F\vert_{\widetilde{Z}})
\end{equation}
for all $q\geq 1$ and $0\leq i\leq r-1$. 

On the other hand, by \cite[Theorem~E]{MP} or \cite[Appendix]{FL}, the condition $\widetilde{\alpha}(Z)>p$ is equivalent to the fact that $Z$ has $(p-1)$-rational singularities:
by definition, this means that, with the above notation, the canonical morphisms
\begin{equation}\label{eq3_prop1}
\Omega_Z^i\to {\mathbf R}{\pi_Z}_*\Omega^i_{\widetilde{Z}}({\rm log}\,F\vert_{\widetilde{Z}})
\end{equation}
are isomorphisms for $0\leq i\leq p-1$. 

We next note that if $\widetilde{\alpha}(Z)>p$, then $r\geq 2p$ by \cite[Proposition~7.4]{MP2}. In particular, we have $r\geq p+1$. 
In fact, this latter inequality also holds if we assume the condition in i) and the weak version of ii) mentioned in Remark~\ref{codim_sing_locus}. 
Indeed, this is clear if $p\leq 2$ by our assumption. If $p\geq 3$, then arguing by induction on $p$, we may assume that $\widetilde{\alpha}(Z)>p-1$,
in which case another application of \cite[Proposition~7.4]{MP2} gives $r\geq 2p-2\geq p+1$.

Since $r\geq p$, we deduce using the isomorphisms (\ref{eq2_thm_char}) 
that the vanishings in i) hold if and only if
$$R^q{\pi_Z}_*\Omega^i_{\widetilde{Z}}({\rm log}\,F\vert_{\widetilde{Z}})=0\quad\text{for all}\quad q\geq 1, i\leq p-1.$$
This is almost the condition for having $(p-1)$-rational singularities: the only missing condition is that
the canonical maps
\begin{equation}\label{eq4_prop1}
\Omega_Z^i\to {\pi_Z}_*\Omega^i_{\widetilde{Z}}({\rm log}\,F\vert_{\widetilde{Z}})
\end{equation}
are isomorphisms for $i\leq p-1$.
We are left with showing that this condition holds as well when i) and (the weak version of) ii) are satisfied.

This is clear if $p=1$, since in this case we assume $r\geq 2$, hence $Z$ is normal. On the other hand, if
$p\geq 2$, then we already know, by induction, that $\widetilde{\alpha}(Z)>1$, hence 
$Z$ has rational singularities by \cite{Saito-B}.
Therefore,
the morphism (\ref{eq4_prop1}) is an isomorphism
if and only if $\Omega_Z^i$ is reflexive 
(see for example \cite[Lemma~2.8]{MP}). 
By
\cite[Theorem~1.11]{Graf}, the canonical morphism
\begin{equation}\label{morphism_double_dual}
\Omega_Z^i\to (\Omega_Z^i)^{**}
\end{equation}
is injective if $i\leq r-1$ and surjective if $i\leq r-2$. Since $i\leq p-1\leq r-2$,
it follows that (\ref{morphism_double_dual}) is bijective, and thus $\Omega_X^i$ 
is reflexive for $i\leq p-1$.
This completes the proof of the theorem.
\end{proof}

\begin{rmk}
We can't completely drop condition ii) from the theorem. For example, suppose that $Z$ is a simple normal crossing divisor in $X$, with $Z$ not smooth, in which case $\widetilde{\alpha}(Z)=1$.
However, in this case we have $R^q\pi_*\Omega^i_Y({\rm log}\,E)=0$ for all $i\geq 0$ and all $q\geq 1$ by \cite[Theorem~31.1]{MP3}.
\end{rmk}

\begin{rmk}\label{rmk_push_forward}
If we assume that $\widetilde{\alpha}(Z)\geq 1$, then the pair $(X,Z)$ is log canonical, hence \cite[Theorem~1.5]{GKKP} implies that for every $i\leq n$, the sheaf
$\pi_*\Omega_Y^i({\rm log}\,E)$ is reflexive. Moreover, for every $\alpha\in\Q$, the sheaf $\pi_*\Omega_Y^i({\rm log}\,E)\big(\lfloor \alpha D\rfloor\big)$
is reflexive. Indeed, this follows via the projection formula if we show that for all $\alpha\in [0,1)$, we have
\begin{equation}\label{eq1_reflexive}
\pi_*\Omega_Y^i({\rm log}\,E)=\pi_*\Omega_Y^i({\rm log}\,E)\big(\lfloor \alpha D\rfloor\big).
\end{equation}
In order to see this, let $U$ be an open subset of $X$ such that $\pi$ is an isomorphism over $U$ and ${\rm codim}_X(X\smallsetminus U)\geq 2$.
If $j\colon U\hookrightarrow X$ is the inclusion, let us consider the commutative diagram
$$
\begin{tikzcd}
\pi_*\Omega_Y^i({\rm log}\,E)\rar{\varphi}\dar{u}  & \pi_*\Omega_Y^i({\rm log}\,E)\big(\lfloor \alpha D\rfloor\big)
\dar{v} \\
j_*j^*\pi_*\Omega_Y^i({\rm log}\,E)\rar{\varphi} \rar{\psi} & j_*j^*\pi_*\Omega_Y^i({\rm log}\,E)\big(\lfloor \alpha D\rfloor\big),
\end{tikzcd}
$$
in which all maps are the natural inclusions. Since $\pi_*\Omega_Y^i({\rm log}\,E)$ is reflexive, it follows that 
$u$ is an isomorphism. Furthermore, since $(X,Z)$ is log canonical, it follows that $Z$ is reduced, and thus $\lfloor \alpha D\rfloor=0$ on $\pi^{-1}(U)$,
hence $\psi$ is an isomorphism as well. This implies that $v$ is surjective, and since $\pi_*\Omega_Y^i({\rm log}\,E)\big(\lfloor \alpha D\rfloor\big)$
is clearly torsion-free, $v$ is also injective, hence an isomorphism. 
We thus conclude that both $v$ and $\varphi$ are isomorphisms, and thus (\ref{eq1_reflexive}) holds.

We note that for $p=0$, the same argument works if we only assume $Z$ reduced, since in this case it is clear that $\pi_*\cO_Y=\cO_X$ is reflexive.
\end{rmk}

\section{A filtered resolution of $\psi_{g,\alpha}(\cO_Y)$, when ${\rm div}(g)$ is an SNC divisor}\label{section_resolution}

In this section, we work on a smooth, irreducible, $n$-dimensional complex algebraic variety $Y$. We assume that we have a nonzero $g\in\cO_Y(Y)$
that defines a divisor $D$ supported on a reduced simple normal crossing divisor $E$. By Generic Smoothness, after possibly replacing $Y$ by an open neighborhood 
of $E$, we may and will assume that $g\colon Y\to\A^1$ is smooth over $\A^1\smallsetminus\{0\}$. 

We consider ${\rm dlog}(g)=\tfrac{dg}{g}\in\Gamma\big(Y,\Omega^1_Y({\rm log}\,E)\big)$ and let
$\Omega_{Y/\A^1}^1({\rm log}\,E)$ be defined by the short exact sequence
\begin{equation}\label{eq_def_rel_diff}
0\to \cO_Y\overset{j}\longrightarrow \Omega_Y^1({\rm log}\,E)\to \Omega_{Y/\A^1}^1({\rm log}\,E)\to 0,
\end{equation}
where $j$ is the injective map defined by the global section ${\rm dlog}(g)$. 

Note that $\Omega_{Y/\A^1}^1({\rm log}\,E)$ is a locally free sheaf of rank $n-1$. Indeed, 
this is clear on $Y\smallsetminus E$, since on this open subset $g$ is invertible, so 
$\Omega_{Y/\A^1}^1({\rm log}\,E)$ is the sheaf of relative differentials for the morphism $Y\smallsetminus E\to\A^1\smallsetminus\{0\}$,
which we assume to be smooth. 
On the other hand,
given $P\in E$, we can find an algebraic system of coordinates $y_1,\ldots,y_n$ in an affine neighborhood $U$ of $P$, such that $E\vert_U$ is defined by $(y_1\cdots y_r)$, for some $r\geq 1$. In this case,
we can write $g=u\cdot\prod_{i=1}^ry_i^{a_i}$, with $u$ invertible, so that
$${\rm dlog}(g)=\dlog(u)+\sum_{i=1}^ra_i\dlog(y_i).$$
It is then clear that  $\Omega_{Y/\A^1}^1({\rm log}\,E)$ is free in a suitable neighborhood of $P$, generated by the classes of $\dlog(y_j)$, for $2\leq j\leq r$, and $dy_k$, for $r+1\leq k\leq n$. 

We put $\Omega_{Y/\A^1}^p({\rm log}\,E):=\wedge^p \Omega_{Y/\A^1}^1({\rm log}\,E)$ for every $p\geq 0$.
Note that the short exact sequence (\ref{eq_def_rel_diff}) induces short exact sequences
\begin{equation}\label{eq2_def_rel_diff}
0\to \Omega_{Y/\A^1}^{\ell-1}({\rm log}\,E)\to \Omega_Y^{\ell}({\rm log}\,E)\to \Omega_{Y/\A^1}^{\ell}({\rm log}\,E)\to 0,
\end{equation}
for $1\leq \ell\leq n$, with the injective map being given by $\eta\mapsto  {\rm dlog}(g) \wedge \eta $. 

Let us consider the log de Rham complex of the left $\cD_Y$-module $\cD_Y$ with respect to $E$:
$$0\to \cD_Y\overset{\nabla}\longrightarrow\Omega_Y^1({\rm log}\,E)\otimes_{\cO_Y}\cD_Y\overset{\nabla}\longrightarrow\ldots\overset{\nabla}\longrightarrow
\Omega_Y^n({\rm log}\,E)\otimes_{\cO_Y}\cD_Y\to 0.$$
Recall that if $y_1,\ldots,y_n$ are local coordinates on $Y$, then the map $\nabla$ is given by
\begin{equation}\label{formula_nabla1}
\nabla(\eta\otimes P)=d\eta\otimes P+\sum_{i=1}^n(dy_i\wedge\eta)\otimes \partial_{y_i}P.
\end{equation}
Note that each term in the complex has a natural structure of right $\cD_Y$-module induced by the right $\cD_Y$-module structure of $\cD_Y$
and that the maps $\nabla$ are $\cD_Y$-linear.
Therefore, this is a complex of right $\cD_Y$-modules.

If $G$ is an effective divisor supported on $E$, then we have an induced complex ${\mathcal C}_{G}$:
$$0\to \cO_Y(-G)\otimes_{\cO_Y}\cD_Y\overset{\nabla}\longrightarrow\cO_Y(-G)\otimes_{\cO_Y}\Omega_Y^1({\rm log}\,E)\otimes_{\cO_Y}\cD_Y\overset{\nabla}\longrightarrow\ldots$$
$$\ldots\overset{\nabla}\longrightarrow
\cO_Y(-G)\otimes_{\cO_Y}\Omega_Y^n({\rm log}\,E)\otimes_{\cO_Y}\cD_Y\to 0.$$
Indeed, if $y_1,\ldots,y_n$ are local coordinates on $Y$ such that $E$ is defined by $(y_1\cdots y_r)$ and $h$ is a local equation of $G$, then
using the fact that $d(h\eta)=h\big({\rm dlog}(h)\wedge\eta+d\eta\big)$, it follows from (\ref{formula_nabla1}) that
\begin{equation}\label{formula_nabla2}
\nabla(h\eta\otimes P)=h ({\rm dlog}(h)\wedge\eta)\otimes P+ h d\eta\otimes P+\sum_{i=1}^n h(dy_i\wedge\eta)\otimes\partial_{y_i}P.
\end{equation}
In particular, ${\mathcal C}_G$ is a subcomplex of the log de Rham complex of $\cD_Y$.
Moreover, it is straightforward to check using (\ref{formula_nabla2}) that $\nabla$ induces maps
$$\cO_Y(-G)\otimes_{\cO_Y}\big({\rm dlog}(g)\wedge \Omega_Y^{\ell-1}({\rm log}\,E)\big)\to \cO_Y(-G)\otimes_{\cO_Y}\big({\rm dlog}(g)\wedge \Omega_Y^{\ell}({\rm log}\,E)\big),$$
hence using (\ref{eq2_def_rel_diff}), we see that ${\mathcal C}_G$ induces a complex $\overline{\mathcal C}_G$:
$$0\to \cO_Y(-G)\otimes_{\cO_Y}\cD_Y\longrightarrow\cO_Y(-G)\otimes_{\cO_Y}\Omega_{Y/\A^1}^1({\rm log}\,E)\otimes_{\cO_Y}\cD_Y\longrightarrow\ldots$$
$$\ldots\longrightarrow
\cO_Y(-G)\otimes_{\cO_Y}\Omega_{Y/\A^1}^{n-1}({\rm log}\,E)\otimes_{\cO_Y}\cD_Y\to 0,$$
that we view placed in cohomological degrees $-(n-1),\ldots,0$. 
If we put 
\begin{equation}\label{formula_filtration}
F_{k-n}\overline{{\mathcal C}}_G^{-q}:=\cO_Y(-G)\otimes_{\cO_Y}\Omega_{Y/\A^1}^{n-1-q}({\rm log}\,E)\otimes_{\cO_Y}  F_{k-q-1}\cD_Y\quad\text{for all}\quad k\in\Z,
\end{equation}
then $\overline{{\mathcal C}}_G$ becomes a complex of filtered right $\cD_Y$-modules.

Note that the short exact sequences (\ref{eq2_def_rel_diff}) induce a short exact sequence of complexes
\begin{equation}\label{eq_seq_complexes}
0\to \overline{\mathcal C}_{G}[-1] \xrightarrow{ \dlog(g)\wedge-} \mathcal{C}_{G} \to \overline{\mathcal C}_{G}\to 0
\end{equation}
(we follow the convention that for a complex ${\mathcal C}$ and $q\in\Z$, the differentials in the complex ${\mathcal C}[q]$ are given by multiplying the 
corresponding differentials in ${\mathcal C}$ with $(-1)^q$).
We thus obtain an endomorphism $N_G$ of $\overline{\mathcal{C}}_G$ in $D^b_{\rm coh}(\cD_X)$, such that we have an exact triangle
\begin{equation}\label{eq_ng}
	\overline{\mathcal C}_{G}[-1] \xrightarrow{ \dlog(g)\wedge-} \mathcal{C}_{G} \to \overline{\mathcal C}_{G} \xrightarrow{N_G} \overline{\mathcal C}_{G}.
\end{equation}

The following is the main result of this section. For ease of comparison with the existing results in the literature, we state this result in the setting of right $\cD$-modules.
Recall that $D={\rm div}(g)$ and for every $\alpha\in \Q_{\geq 0}$, let $D_{\alpha}=\lceil \alpha D\rceil$ and $D_{>\alpha}=D_{\alpha+\epsilon}$ for $0<\epsilon\ll 1$,
so that $D_{>\alpha}=\lfloor \alpha D\rfloor+E$. 

\begin{thm}\label{thm_resolution}
With the above notation, for every $\alpha\in \Q_{>0}$, the following hold:
\begin{enumerate}
	\item[i)] The complex $\overline{\mathcal C}_{D_{\alpha}}$ gives a resolution of $\big(V_{-\alpha}B_g^r,F_{\bullet}[-1]\big)$ in the category of filtered right $\cD_Y$-modules.
	\item[ii)] The quotient
$\overline{\mathcal C}_{D_{\alpha}}/\overline{\mathcal C}_{D_{>\alpha}}$, with the induced filtration, gives a resolution of the $\cD_Y$-module $\big(\gr_{-\alpha}^VB_g^r, F_{\bullet}[-1]\big)$
in the category of filtered right $\cD_Y$-modules.
	\item[iii)]  The endomorphism $N_{D_\alpha}$ induces an operator $N_\alpha$ on the $0$-th cohomology of $\overline{\mathcal C}_{D_{\alpha}}$, which gets
identified with the endomorphism $-\theta$ on $V_{-\alpha}B^r_g$.
\end{enumerate}
\end{thm}

\begin{rmk}
	Steenbrink already noticed in his seminal work~\cite{Steenbrink} that the quotient complex $\overline{\cC}_{D_0}/\overline{\cC}_{D_{>0}}$ is quasi-isomorphic to the perverse nearby cycles of $\C_Y[n]$
	with respect to $g$ (though not respecting the Hodge filtration), when $D={\rm div}(g)$ is reduced. Theorem~\ref{thm_resolution} gives an analogue for the nearby
	cycles with respect to other eigenvalues, without assuming that $D$ is reduced.
\end{rmk}

\begin{rmk}
The reader may compare Theorem~\ref{thm_resolution} with \cite[Proposition~15.12.5]{MHM_project}, which also gives a resolution of each $V_{-\alpha}B_g^r$.
After choosing local coordinates, the resolution in Theorem~\ref{thm_resolution} can be identified with the one in \emph{loc. cit}. The main difference in the two approaches is that on one side, we keep track of filtrations and, on the other side, we have a global description of the resolution. We note that this latter point is important
for getting a description of the direct image.
\end{rmk}

\begin{rmk}
In connection with the assertion ii) in the theorem, note that if $\alpha\in (0,1]$, then the filtered right $\cD_Y$-module $\big(\gr_{-\alpha}^VB_g^r, F_{\bullet}[-1]\big)$
is precisely $\psi_{g,-\alpha}(\omega_Y)$ (see Remark~\ref{right_nearby_cycles}).
\end{rmk}

\begin{rmk}[Right-left conversion]\label{rmk_conversion} 
Since in the next section we will return to left $\cD$-modules, we translate the assertions in Theorem~\ref{thm_resolution} to that setting.
Note first that for any effective divisor $G$ supported on $E$, the complex of left $\cD_Y$-modules
$\widetilde{\cC}_G$ corresponding to $\overline{\cC}_G$ is given by
$$0\to \cD_Y\otimes_{\cO_Y}\cO_Y(-G)\otimes_{\cO_Y}\omega_Y^{-1}\longrightarrow\cD_Y\otimes_{\cO_Y}\cO_Y(-G)\otimes_{\cO_Y}\Omega_{Y/\A^1}^1({\rm log}\,E)\otimes_{\cO_Y}\omega_Y^{-1}\longrightarrow\ldots$$
$$\ldots\longrightarrow
\cD_Y\otimes_{\cO_Y}\cO_Y(-G)\otimes_{\cO_Y}\Omega_{Y/\A^1}^{n-1}({\rm log}\,E)\otimes_{\cO_Y}\omega_Y^{-1}\to 0.$$
This is a complex of filtered $\cD_Y$-modules, with
$$F_k\widetilde{\cC}_G^{-q}=F_{k-q-1}\cD_Y\otimes_{\cO_Y}\cO_Y(-G)\otimes_{\cO_Y}\Omega_{Y/\A^1}^{n-1-q}({\rm log}\,E)\otimes_{\cO_Y}\omega_Y^{-1}\quad\text{for all}\quad k\in\Z.$$
For every $\alpha\in\Q_{>0}$, the assertions in Theorem~\ref{thm_resolution} become:
\begin{enumerate}
	\item[i)] The complex $\widetilde{\mathcal C}_{D_{\alpha}}$ gives a resolution of $V^{\alpha}B_g$ in the category of filtered left $\cD_Y$-modules.
	\item[2)] The quotient
$\widetilde{\mathcal C}_{D_{\alpha}}/\widetilde{\mathcal C}_{D_{>\alpha}}$, with the induced filtration, gives a filtered resolution of $\gr^{\alpha}_VB_g$
(which is isomorphic to $\psi_{g,\alpha}(\cO_Y)$ for $\alpha\in (0,1]$).
	\item[3)] The endomorphism $N_{D_\alpha}$ induces an operator $N_\alpha$ on the $0$-th cohomology of $\widetilde{\mathcal C}_{D_{\alpha}}$, which gets identified with the endomorphism $\partial_tt$ on $V^{\alpha}B_g$.
\end{enumerate}
\end{rmk}

\begin{proof}[Proof of Theorem~\ref{thm_resolution}]
The fact that for $\alpha\in \Q \cap [0,1)$,  the complexes $\overline{\mathcal C}_{D_{\alpha}}$ and $\overline{\mathcal C}_{D_{\alpha}}{/}\overline{\mathcal C}_{D_{>\alpha}}$
are filtered quasi-isomorphic to filtered $\cD_Y$-modules (placed in cohomological degree $0$) was shown in 
\cite{LMH}*{Theorem 7.1 and 7.2}. However, we give a self-contained direct proof of this assertion below, for all $\alpha>0$, and also describe the degree $0$ cohomology of these complexes. 

A key fact is the following description of the $V$-filtration in the case of a simple normal crossing divisor, for which we refer to \cite[Proposition~3.5]{Saito-MHM}
(see also \cite[Section~2]{Budur_Saito})\footnote{We note that the proof in \cite{Saito-MHM} is done in the analytic setting, assuming that $g$ is a monomial 
$y_1^{a_1}\cdots y_n^{a_n}$.
In the algebraic setting, we can find local algebraic coordinates $y_1,\ldots,y_n$ such that $g=uh$, where $u$ is invertible and $h=y_1^{a_1}\cdots y_n^{a_n}$, with $a_1,\ldots,a_n\in\Z_{\geq 0}$. It is standard to see that
the formula in \cite{Saito-MHM} implies that the $V$-filtration with respect to $h$ in the algebraic setting also satisfies (\ref{eq1_descr_Vfiltration}). On the other hand, the $V$-filtrations with respect to $g$ and $h$
determine each other, as described for example in \cite[Remark~5.10]{CDM}. We thus conclude that the $V$-filtration with respect to $g$ also satisfies (\ref{eq1_descr_Vfiltration}).}. This says that for every $\alpha>0$ and every $p\in\Z$, we have
\begin{equation}\label{eq1_descr_Vfiltration}
F_pV_{-\alpha}B_g^r=\omega_Y(-D_{\alpha}+E)\delta\cdot F_{p+n}\big(\cD_Y[\theta]\big),
\end{equation}
where $F_q\big(\cD_Y[\theta]\big)=\sum_{0\leq i\leq q}F_{q-i}\cD_Y\theta^i$ (recall that $\theta=t\partial_t$).

Let us define an augmented version of the complex $\overline{\mathcal C}_{D_{\alpha}}$, by defining a $\cD_Y$-linear map $\sigma_{\alpha}\colon \overline{\mathcal C}_{D_{\alpha}}^0\to V_{-\alpha}B_g^r$, as follows:
\[
\begin{tikzcd}[column sep=small, row sep=tiny]
\sigma_\alpha \, \colon	&	\overline{\mathcal{C}}_\alpha^0 \arrow[r,"\cong"] & \omega_{Y}(E-D_\alpha) \otimes_{\cO_Y} \cD_Y \arrow[r] &  V_{-\alpha}B_g^r \\
	&	\overline{\xi}\otimes_{\cO_Y}P \arrow[r,mapsto] &  \dlog (g)\wedge \xi \otimes_{\cO_Y}P\arrow[r,mapsto] & \big( \dlog (g) \wedge \xi\big)\delta \cdot P,
\end{tikzcd}
\]
for any local section $\xi$ of $\cO_{Y}(-D_\alpha)\otimes_{\cO_Y}\Omega^{n-1}_{Y}(\log\, E)$ and differential operator $P$,
where we write $\overline{\xi}$ for the class of $\xi$ in $\cO_{Y}(-D_\alpha)\otimes_{\cO_Y}\Omega^{n-1}_{Y/\A^1}(\log\, E)$.
The fact that the first map is an isomorphism follows from the exact sequence (\ref{eq2_def_rel_diff}) (for $\ell=n$). Note that by formula (\ref{eq1_descr_Vfiltration}),
the image of the second map lies indeed in $V_{-\alpha}B_g^r$, and the map is compatible with the filtrations if on $V_{-\alpha}B_g^r$ we consider the
filtration $F_{\bullet}[-1]$.

\begin{lem}
With the above notation, the composition
$$\overline{\mathcal C}_{D_{\alpha}}^{-1}\longrightarrow \overline{\mathcal C}_{D_{\alpha}}^0\overset{\sigma_{\alpha}}\longrightarrow V_{-\alpha}B_g^r$$
is 0.
\end{lem}

\begin{proof}
By viewing $\overline{\mathcal C}_{D_{\alpha}}[-1]\hookrightarrow \mathcal{C}_{D_{\alpha}}$ via (\ref{eq_seq_complexes}), we see that it is enough to show that 
for every $\eta\in\cO_Y(-D_{\alpha})\otimes_{\cO_Y}\Omega_Y^{n-2}({\rm log}\,E)$ and $P\in\cD_Y$, the composition
$$\cO_Y(-D_{\alpha})\otimes_{\cO_Y}\Omega^{n-1}_Y({\rm log}\,E)\otimes_{\cO_Y}\cD_Y\to \cO_Y(-D_{\alpha})\otimes_{\cO_Y}\Omega^n_Y({\rm log}\,E)\otimes_{\cO_Y}\cD_Y\to B^r_g$$
maps $u=\big(\dlog(g)\wedge\eta\big)\otimes P$ to $0$. Since both maps are $\cD_Y$-linear, we may assume that $P=1$. We may also assume that we work in an open subset $U$ with
algebraic coordinates $y_1,\ldots,y_n$, so by definition of the two maps, $u$ is mapped to 
\begin{equation}\label{eq_lem1}
d\big(\dlog(g)\wedge\eta\big)\delta+\sum_{i=1}^n \big((dy_i\wedge\dlog(g)\wedge\eta)\delta\big)\cdot\partial_{y_i}.
\end{equation}
Using the description of the left $\cD_Y$-action on $B_g$ in (\ref{eq1_action_on_B_f})
and the formula for how such an action is turned into a right $\cD_Y$-action in (\ref{eq_description_right_action}), we see that
$$\sum_{i=1}^n \big((dy_i\wedge\dlog(g)\wedge\eta)\delta\big)\cdot\partial_{y_i}=
\sum_{i=1}^n\big(dy_i\wedge\dlog(g)\wedge\eta)\cdot\partial_{y_i}\big)\delta+\sum_{i=1}^n\tfrac{\partial g}{\partial y_i}(dy_i\wedge\dlog(g)\wedge\eta)\partial_t\delta.$$
By (\ref{eq_identity_omega}), we have
$$\sum_{i=1}^n\big(dy_i\wedge\dlog(g)\wedge\eta)\cdot\partial_{y_i}\big)\delta=-d\big(\dlog(g)\wedge\eta\big)\delta,$$
while 
$$\sum_{i=1}^n\tfrac{\partial g}{\partial y_i}(dy_i\wedge\dlog(g)\wedge\eta)\partial_t\delta=\big(dg\wedge\dlog(g)\wedge\eta)\partial_t\delta=0.$$
We thus conclude that the expression in (\ref{eq_lem1}) is $0$, which concludes the proof of the lemma.
\end{proof}

We now return to the proof of the theorem. By the lemma, the complex $\overline{\mathcal C}_{D_{\alpha}}$, together with the map $\sigma_{\alpha}$, gives
an augmented filtered complex, that we denote by $\overline{\mathcal C}_{D_{\alpha}}^{\rm aug}$. 
We will show that, for all $p\in \Z$, the induced complex ${\rm Gr}_p^F\overline{\mathcal C}_{D_{\alpha}}^{\rm aug}$ is acyclic. By a standard inductive argument, this gives the assertion in i).

The fact that each ${\rm Gr}_p^F\overline{\mathcal C}_{D_{\alpha}}^{\rm aug}$ is acyclic can be checked locally in the Zariski topology. Moreover, it can be checked analytically locally: this follows from the fact that for every point $P\in Y$, the morphism
of local rings 
\begin{equation}\label{eq_local_rings}
\cO_{Y,P}\to\cO_{Y^{\rm an},P}
\end{equation} is faithfully flat. Indeed, this implies that the functor ${\mathcal F}\mapsto {\mathcal F}^{\rm an}$ from algebraic coherent sheaves on $Y$ to analytic coherent sheaves 
on $Y^{\rm an}$ is exact and ${\mathcal F}=0$ if and only if ${\mathcal F}^{\rm an}=0$.
We treat the cases $P\in E$ and $P\not\in E$ separately. Suppose first that $P\in E$. 
The advantage in working analytically is that in this case, in some chart $U$ around $P$, we can choose
analytic coordinates $y_1,\ldots,y_n$ centered at $P$ such that 
$g=y_1^{a_1}\cdots y_r^{a_r}$, for some $r\geq 1$ and $a_1,\ldots,a_r>0$ (working algebraically, we can only insure that $g=uy_1^{a_1}\cdots y_r^{a_r}$
for some invertible function $u$). 
On the other hand, the map $(y_1,\ldots,y_n)\colon U\to {\mathbf C}^n$ is biholomorphic onto an open ball in ${\mathbf C}^n$, so it is enough to prove the assertion we need on 
${\mathbf C}^n$.
Once we are in this particular situation, it is convenient to switch back to the algebraic setting: note that using again the faithful flatness of the homomorphisms (\ref{eq_local_rings}), we see that proving the desired assertion in the algebraic setting implies it in the analytic setting as well. 
Hence, from now on, we assume that $Y={\mathbf A}^n$, with coordinates $y_1,\ldots,y_n$, and $g=y_1^{a_1}\cdots y_r^{a_r}$. 
We put $dy=dy_1\wedge\ldots\wedge dy_n$.

Note that $s_\alpha=y_1^{\lceil \alpha a_1 \rceil}\cdots y_r^{\lceil \alpha a_r \rceil}$ gives an isomorphism $\cO_Y(-D_{\alpha})\simeq \cO_Y$ and we have 
a trivialization of $\Omega^{1}_{Y/\A^1}(\log\, E)$ by the classes of 
\[
a_2\cdot \dlog (y_2), \dots, a_r \cdot \dlog (y_r), d y_{r+1},\dots,d y_n.
\] 
A straightforward computation then shows that the complex $\overline{\mathcal C}_{D_\alpha}$ can be identified to the Koszul complex  of the left multiplications by the differential operators 
\begin{equation}\label{eq_ann}
	\tfrac{y_2\partial_{y_2}}{a_2}-\tfrac{y_1\partial_{y_1}}{a_1}+\tfrac{\lceil \alpha a_2 \rceil}{a_2}-\tfrac{\lceil \alpha a_1 \rceil}{a_1},\cdots, \tfrac{y_r\partial_{y_r}}{a_r}-\tfrac{y_1\partial_{y_1}}{a_1}+\tfrac{\lceil \alpha a_r \rceil}{a_r}-\tfrac{\lceil \alpha a_1 \rceil}{a_1}, \partial_{y_{r+1}},\dots, \partial_{y_n}
\end{equation}
on the sheaf $\cD_Y$. The principal symbols of the above differential operators form a regular sequence in $\gr^F_\bullet \cD_Y\simeq\cO_Y[z_1,\ldots,z_n]$ (where we write $z_i$
for the class of $\partial_{y_i}$), and thus, using also (\ref{formula_filtration}), we see that 
$\gr^F_p\overline{\mathcal C}_{D_\alpha}$ is a resolution of 
\begin{equation}\label{eq_descr_H0}
\big(\cO_Y[z_1,\ldots,z_r]/(Q_2,\ldots,Q_r)\big)_{p+n-1},\quad\text{where}\quad Q_i=\tfrac{y_iz_i}{a_i}-\tfrac{y_1z_1}{a_1}\quad\text{for}\quad 2\leq i\leq r.
\end{equation}
Here, we consider $\cO_Y[z_1,\ldots,z_r]$ graded by the total degree in the $z$ variables.

Since $Y=\A^n$ and $g$ is a monomial, we may and will consider $B^r_g$ as a right ${\mathbf Z}^n$-graded module over the Weyl algebra 
$A=\C\langle y_1,\ldots,y_n,t,\partial_{y_1},\ldots,\partial_{y_n},\partial_t\rangle$, where ${\rm deg}(dy\delta)=0$ 
and the grading of $A$ is given by 
${\rm deg}(y_i)=e_i=-{\rm deg}(\partial_{y_i})$
for $1\leq i\leq n$ (here $e_1,\ldots,e_n$ is the standard basis of $\Z^n$), while ${\rm deg}(t)=\sum_{i=1}^ra_ie_i=-{\rm deg}(\partial_t)$. 
Note that formula (\ref{eq1_descr_Vfiltration}) implies that each $F_pV_{-\alpha}B_g^r$ is a $\Z^n$-graded $\C[y_1,\ldots,y_n]$-submodule of $B_g^r$.
For $v=(v_1,\ldots,v_n)\in\Z_{\geq 0}^n$, we put $|v|=\sum_iv_i$, $y^v=y_1^{v_1}\cdots y_n^{v_n}$, and $\partial_y^v=\partial_{y_1}^{v_1}\cdots\partial_{y_n}^{v_n}$.

Using the formulas (\ref{eq1_action_on_B_f}) and (\ref{eq2_action_on_B_f}) describing the left $A$-module structure on $B_g$ and formula (\ref{eq_description_right_action})
allowing us to get the corresponding right structure on $B_g^r$, it is easy to see that for every $v=(v_1,\ldots,v_n)\in\Z_{\geq 0}^n$ and $m\in\Z_{\geq 0}$, we have
\begin{equation}\label{eq1_Bg}
y^vdy\partial_t^m\delta\cdot (\partial_{y_i}y_i+v_i)=y^vdy\partial_t^m\delta\cdot a_i(-\theta+m)\quad\text{for}\quad 1\leq i\leq r\quad\text{and}
\end{equation}
\begin{equation}\label{eq2_Bg}
y^vdy\partial_t^m\delta\cdot\partial_{y_i}=-v_iy^{v-e_i}dy\partial_t^m\delta\quad\text{for}\quad r<i\leq n.
\end{equation}

We deduce from (\ref{eq1_Bg}), (\ref{eq2_Bg}), and (\ref{eq1_descr_Vfiltration}) that if $b\in\Z^n_{\geq 0}$ is given by $b_i=\lceil\alpha a_i\rceil-1$ for $i\leq r$ and $b_i=0$ for $i>r$, then 
\begin{equation}\label{eq2_descr_Vfiltration}
F_pV_{-\alpha}B_g^r=\bigoplus_{v,w} y^bdy\delta\cdot y^v\partial_y^w\C[\theta]_{\leq p+n-|w|},
\end{equation}
where the sum is over those $v,w\in \Z_{\geq 0}^n$ with $v_iw_i=0$ for $i\leq r$ and $w_i=0$ for all $i>r$. Here we write $\C[\theta]_{\leq m}$
for the vector space of polynomials in $\theta$ of degree $\leq m$. 
Indeed, it follows from 
(\ref{eq1_descr_Vfiltration}) that
$$F_pV_{-\alpha}B_g^r=\sum_{v,w,j}y^bdy\delta\cdot {\mathbf C}y^v\partial_y^w\theta^j,$$
where the sum is over all $v,w\in \Z_{\geq 0}^n$ and all $j$, with $0\leq j\leq p+n-|w|$. Using the formula (\ref{eq1_Bg}), we see
that it is enough to take the sum over those $v$ and $w$ with $v_iw_i=0$ for $i\leq r$, and using the formula (\ref{eq2_Bg}),
we see that that we may also assume that $w_i=0$ for all $i>r$ (since $y^bdy\delta\cdot\partial_{y_i}=0$ for $i>r$).
The fact that in (\ref{eq2_descr_Vfiltration}) we have a direct sum
 follows
from the fact that each component lies in a different degree in $\Z^n$. 
Moreover, for every nonzero $u\in B_g^r$ and every $m\in\Z_{\geq 0}$, it is easy to see that $u,u\theta,\ldots,u\theta^m$ are linearly independent over $\C$: it is enough to 
consider the first piece of the Hodge filtration that contains a linear combination $\sum_{j=0}^mc_ju\theta^j$. 
Using again (\ref{eq1_Bg}) to replace the action of $\theta$ by that of $\partial_{y_1}y_1$, we
deduce from (\ref{eq2_descr_Vfiltration}) that
\begin{equation}\label{eq3_descr_Vfiltration}
\gr^F_pV_{-\alpha}B_g^r=\bigoplus_{v,w}\C y^bdy\delta\cdot y^v\partial_y^w,
\end{equation}
where $v,w\in\Z_{\geq 0}^n$, with $w_i=0$ for $i>r$, $v_iw_i=0$ for $2\leq i\leq r$, and $|w|=n+p$.

Let's describe now, via 
the identification of $\cH^0\big(\gr^F_p\overline{\mathcal C}_{D_\alpha}\big)$
with (\ref{eq_descr_H0}), the map 
$$\overline{\sigma_{\alpha}}\colon \cH^0\big(\gr^F_p\overline{\mathcal C}_{D_\alpha}\big)\to \gr^F_{p-1}V_{-\alpha}B_g^r$$ induced by $\sigma_{\alpha}$.
Using the relations $Q_2,\ldots,Q_r$, it is easy to see that
$$\big(\cO_Y[z_1,\ldots,z_r]/(Q_2,\ldots,Q_r)\big)_{p+n-1}=\bigoplus_{v,w}\C y^vz^w,$$
where the sum is over $v,w\in\Z_{\geq 0}^n$, with $v_iw_i=0$ for $2\leq i\leq r$ and $|w|=p+n-1$. 
By definition,
$\overline{\sigma_{\alpha}}(y^vz^w)$ is the class of 
$$s_{\alpha}\dlog(g)\wedge a_2\dlog(y_2)\wedge\ldots\wedge a_r\dlog(y_r)\wedge dy_{r+1}\wedge\ldots\wedge dy_n\delta\cdot y^v\partial_y^w=
a_1\cdots a_r y^bdy\delta\cdot y^v\partial_y^w.$$
Our description of $\gr^F_pV_{-\alpha}B_g^r$ thus implies that $\overline{\sigma_{\alpha}}$ gives an isomorphism 
$$\cH^0\big(\gr^F_p\overline{\mathcal C}_{D_\alpha}\big)\simeq \gr^F_{p-1}V_{-\alpha}B_g^r.$$

In order to complete the proof of i), we still need to show that in the case of an arbitrary $Y$, the assertion also holds on $Y\smallsetminus E$.
We may thus assume that $E=0$, that is, $g$ is invertible. By our running assumption, in this case $g\colon Y\to\A^1\smallsetminus\{0\}$ is a smooth morphism. 
Passing through the analytic setting, as before, we reduce to the case when $Y=\A^{n-1}\times \big(\A^1\smallsetminus\{0\}\big)$ and $g\colon Y\to\A^1$
is the projection onto the last component. Note that if we denote the coordinates on $Y$ by $y_1,\ldots,y_n$, then 
$$\overline{\mathcal C}_{D_{\alpha}}=\Omega_{Y/\A^1}^{\bullet}\otimes_{\cO_Y}\cD_Y=\big(\Omega_{\A^{n-1}}^{\bullet}\otimes_{\cO_{\A^{n-1}}}\cD_{\A^{n-1}}\big)[y_n,y_n^{-1},\partial_{y_n}].$$
It is now clear that this is quasi-isomorphic to $\omega_{\A^{n-1}}[y_n,y_n^{-1},\partial_{y_n}]$, which is (filtered) isomorphic via $\sigma_{\alpha}$ to $\big(B_g^r,F_{\bullet}[-1]\big)$
(note that in this case we have $V_{-\alpha}B_g^r=B_g^r$).
This completes the proof of i). 

Since we have a short exact sequence of complexes
$$0\to {\rm Gr}^F_{\bullet}(\overline{\mathcal C}_{D_{>\alpha}})\to \gr^F_{\bullet}(\overline{\mathcal C}_{D_{\alpha}})\to \gr^F_{\bullet}(\overline{\mathcal C}_{D_{\alpha}}/\overline{\mathcal C}_{D_{>\alpha}})\to 0,$$
by taking the long exact sequence in cohomology and applying i) for $\alpha$ and $\alpha+\epsilon$ (with $0<\epsilon\ll 1$), we obtain the assertion in ii).

We now prove the last assertion. Again, there are two cases to consider: when we are at a point in $E$ and when we are on $Y\smallsetminus E$. We treat the former case and leave the latter one
as an exercise for the reader. For the case of a point in $E$, 
arguing as before, we see that 
it is enough to treat the case when $Y=\A^n$ and $g=y_1^{a_1}\cdots y_r^{a_r}$. Note first that by what we have already proved, $\cH^0(\overline{\mathcal C}_{D_{\alpha}})$ is generated over $\cD_Y$ by the class $\overline{u}$ of the section 
$$u=\big(s_{\alpha} a_2 \dlog(y_2)\wedge\ldots\wedge a_r\dlog(y_r)\wedge dy_{r+1}\wedge\ldots\wedge dy_n\big)\otimes 1$$
of $\cO_Y(-D_{\alpha})\otimes_{\cO_Y}\Omega_{Y}^{n-1}({\rm log}\,E)\otimes_{\cO_Y}\cD_Y$. Indeed, we know that $\overline{\sigma_{\alpha}}$ gives an isomorphism
$\cH^0(\overline{\mathcal C}_{D_{\alpha}})\simeq V_{-\alpha}B^r_g$ and $\sigma_{\alpha}(u)=s_{\alpha}\tfrac{a_1\cdots\a_r}{y_1\cdots y_r}dy\delta$;
by (\ref{eq1_descr_Vfiltration}), this element generates $V_{-\alpha}B^r_g$ over $\cD_Y[\theta]$, and in fact over $\cD_Y$ by (\ref{eq1_Bg}).

Since both maps in iii) are $\cD_Y$-linear, we see that it is enough to show that they take the same value on $\overline{u}$. In order to compute $N_{D_{\alpha}}(\overline{u})$, we apply the differential of ${\mathcal C}_{D_{\alpha}}$ to $u$ to get $s_{\alpha}\eta\cdot\big(\tfrac{\lceil \alpha a_1\rceil}{a_1}+\tfrac{1}{a_1}y_1\partial_{y_1}\big)$, where
$$\eta=a_1\dlog(y_1)\wedge\ldots\wedge a_r\dlog(y_r)\wedge dy_{r+1}\wedge\ldots\wedge dy_n.$$
Since 
$\dlog(g)\wedge u=s_{\alpha}\eta$,
we conclude that $N_{D_{\alpha}}$ maps $\overline{u}$ to $\overline{u}\cdot \big(\tfrac{\lceil \alpha a_1\rceil}{a_1}+\tfrac{1}{a_1}y_1\partial_{y_1}\big)$.

On the other hand, the isomorphism induced by $\sigma_{\alpha}$ maps $\overline{u}$ to $\tfrac{s_{\alpha}a_1\cdots a_r}{y_1\cdots y_r}dy\delta$. By (\ref{eq1_Bg}), we have
$$\tfrac{s_{\alpha}a_1\cdots a_r}{y_1\cdots y_r}dy\delta\cdot \big(\tfrac{\lceil \alpha a_1\rceil}{a_1}+\tfrac{1}{a_1}y_1\partial_{y_1}\big)=
\tfrac{s_{\alpha}a_1\cdots a_r}{y_1\cdots y_r}dy\delta\cdot \big(\tfrac{\lceil \alpha a_1\rceil-1}{a_1}+\tfrac{1}{a_1}\partial_{y_1}y_1\big)=\tfrac{s_{\alpha}a_1\cdots a_r}{y_1\cdots y_r}dy\delta\cdot (-\theta).$$
This completes the proof of iii).
\end{proof}

\begin{rmk}
It is clear from the proof of Theorem~\ref{thm_resolution} that for every $0<\alpha<\alpha'$, the inclusion $V_{-\alpha'}B_g^r\hookrightarrow V_{-\alpha}B_g^r$ corresponds to the natural inclusion
$\overline{\mathcal C}_{D_{\alpha'}}\hookrightarrow \overline{\mathcal C}_{D_{\alpha}}$ induced by the inclusion $\cO_Y(-D_{\alpha'})\hookrightarrow\cO_Y(-D_{\alpha})$. We also note that, arguing 
as in the proof of assertion iii) in Theorem~\ref{thm_resolution}, one can check that for every $\alpha>0$, the map $V_{-\alpha}B_g^r\overset{\cdot t}\longrightarrow V_{-\alpha-1}B_g^r$ (which is
an isomorphism) corresponds to the isomorphism $\overline{\mathcal C}_{D_{\alpha}}\to \overline{\mathcal C}_{D_{\alpha+1}}$ which is induced by the isomorphism
$\cO_Y(-D_{\alpha})\hookrightarrow\cO_Y(-D_{\alpha+1})$ given by multiplication with $g$.
\end{rmk}

\section{The description of the minimal exponent}\label{section_final}

Suppose that $X$ is a smooth, irreducible, $n$-dimensional complex algebraic variety and $Z$ is a nonempty hypersurface in $X$.
We fix a log resolution $\pi\colon Y\to X$ of $(X,Z)$ that is an isomorphism over $X\smallsetminus Z$ and put $E=\pi^*(Z)_{\rm red}$.
We put $D=\pi^*(Z)$ and for every $\alpha\in \Q_{\geq 0}$, we write $D_{\alpha}=\lceil \alpha D\rceil$ and $D_{>\alpha}=\lfloor \alpha D\rfloor+E=D_{\alpha+\epsilon}$, 
where $0<\epsilon\ll 1$.

We begin with the following consequence of the description of a filtered resolution of the nearby cycles in the SNC case. 

\begin{cor}\label{cor1_resolution}
If $Z$ is defined by $f\in\cO_X(X)$, then for every
$\alpha\in \Q\cap (0,1]$ and every $i\in\Z$, with the notation in Theorem~\ref{thm_resolution}, we have an isomorphism
$${\rm GR}_{i-n+1}^F{\rm DR}_X\psi_{f,\alpha}(\cO_X)\simeq {\mathbf R}\pi_*\big((\cO_Y(-D_{\alpha})/\cO_Y(-D_{>\alpha}))
\otimes_{\cO_Y}\Omega_{Y/\A^1}^{n-1-i}({\rm log}\,E)\big)[i].$$
\end{cor}

\begin{proof}
Let $g=f\circ \pi$, so $D={\rm div}(g)$.
Since both sides of the formula in the theorem are supported on $Z$, it follows that we may replace $X$ by an open neighborhood of $Z$ to assume that 
$X\overset{f}\longrightarrow \A^1$ is smooth over $\A^1\smallsetminus\{0\}$, and thus $Y\overset{g}\longrightarrow\A^1$ is smooth
over $\A^1\smallsetminus\{0\}$. Therefore, we can apply Theorem~\ref{thm_resolution} for $Y$ and $g$.

With the notation in Remark~\ref{rmk_conversion},
note that for every $q$, we have an isomorphism of filtered left $\cD_Y$-modules
\begin{equation}\label{eq_quot_C}
\widetilde{\mathcal C}^{-q}_{D_{\alpha}}/\widetilde{\mathcal C}^{-q}_{D_{>\alpha}}\simeq \cD_Y(q+1)\otimes_{\cO_Y}\big(\cO_Y(-D_{\alpha})/\cO_Y(-D_{>\alpha})\big)
\otimes_{\cO_Y}\Omega_{Y/\A^1}^{n-1-q}({\rm log}\,E)\otimes_{\cO_Y}\omega_Y^{-1},
\end{equation}
where $\cD_Y(q+1)=\big(\cD_Y,F[-q-1]\big)$. 
Recall that the Spencer complex on $Y$ (which is the complex of left $\cD_Y$-modules corresponding to ${\rm DR}_Y(\cD_Y)$),
gives a filtered resolution of $\cO_Y$, which implies that 
$${\rm Gr}^F_{i-n}{\rm DR}_Y(\cD_Y)\simeq {\rm Gr}^F_{i-n}(\omega_Y)=\left\{
\begin{array}{cl}
\omega_Y, & \text{if}\,\,i=0; \\[2mm]
0, & \text{otherwise}.
\end{array}\right.$$
We thus deduce from (\ref{eq_quot_C}) that
$${\rm Gr}^F_{i-n+1}{\rm DR}_Y\big(\widetilde{\mathcal C}^{-q}_{D_{\alpha}}/\widetilde{\mathcal C}^{-q}_{D_{>\alpha}}\big)\simeq\left\{
\begin{array}{cl}
\Omega^{n-1-q}_{Y/\A^1}({\rm log}\,E)\otimes_{\cO_Y}\cO_Y(-D_{\alpha})/\cO_Y(-D_{>\alpha}), & \text{if}\,\,q=i; \\[2mm]
0, & \text{otherwise}.
\end{array}\right.$$
Using the fact that ${\rm Gr}^F_{i-n+1}{\rm DR}_Y$ maps short exact sequences of filtered $\cD_Y$-modules (with strict morphisms) to short exact sequences of complexes, 
and the fact that by 
Theorem~\ref{thm_resolution} (see also Remark~\ref{rmk_conversion}),
$\widetilde{\mathcal C}_{D_{\alpha}}/\widetilde{\mathcal C}_{D_{>\alpha}}$ gives a filtered resolution of $\psi_{g,\alpha}(\cO_Y)$,
we conclude that
$${\rm Gr}_{i-n+1}^F{\rm DR}_Y\psi_{g,\alpha}(\cO_Y)\simeq \big(\cO_Y(-D_{\alpha})/\cO_Y(-D_{>\alpha})\big)
\otimes_{\cO_Y}\Omega_{Y/\A^1}^{n-1-i}({\rm log}\,E)[i].$$
The assertion in the corollary now follows from Proposition~\ref{prop_birat_nearby_cycles}.
\end{proof}

We also obtain the following vanishing result:

\begin{cor}\label{cor3_resolution}
For every $\beta\geq 0$, the following hold:
\begin{enumerate}
\item[i)] $R^j\pi_*\Omega_Y^i({\rm log}\,E)(-E-\lfloor\beta D\rfloor)=0$ if $i+j>n$.
\item[ii)] $R^j\pi_*\Omega_{Y/\A^1}^i({\rm log}\,E)(-E-\lfloor\beta D\rfloor)=0$ if $i+j>n-1$.
\end{enumerate}
Note that in ii), we assume that $Z$ is defined by $f\in\cO_X(X)$ and $X\smallsetminus Z\overset{f}\longrightarrow \A^1\smallsetminus\{0\}$
is smooth.
\end{cor}

\begin{proof}
Let's first prove the assertion in i). Since the assertion is local on $X$ and it holds trivially on $X\smallsetminus Z$, 
we may and will assume that $Z$ is defined by $f\in\cO_X(X)$ and the map $X\smallsetminus Z\overset{f}\longrightarrow \A^1\smallsetminus\{0\}$
is smooth.
When $\beta=0$, the assertion is Steenbrink Vanishing, see \cite[Theorem~2]{Steenbrink85} (in our setting of hypersurfaces, this is easy to prove,
see \cite[Section~2.6]{MOPW}). Furthermore, it follows from the projection formula that the assertion holds when $\beta=1$ and it is enough
to show that it also holds for all $\beta\in (0,1)$. Since $E+\lfloor \beta D\rfloor=D_{>\beta}=D_{\beta+\epsilon}$ for $0<\epsilon\ll 1$, it is enough to show that
\begin{equation}\label{eq_ETS}
R^j\pi_*\Omega_Y^i({\rm log}\,E)(-D_{\alpha})=0 \quad\text{for}\quad i+j>n
\end{equation}
whenever $\alpha\in (0,1]$. Moreover,
we may assume that (\ref{eq_ETS}) holds when replacing $D_{\alpha}$ by $D_{>\alpha}$: indeed, we have seen that we know this when $\alpha=1$, and for $\alpha<1$ we use the fact the $D_{\gamma}$ change value for a discrete set of parameters $\gamma$ (more precisely,
there is a positive integer $\ell$ such
that $D_{\gamma}$ is constant for $\gamma\in\big(i/\ell,(i+1)/\ell\big]$, with $i\in {\mathbf Z}$).
By tensoring the short exact sequence
\begin{equation}\label{eq_SES1}
0\to\cO_Y(-D_{>\alpha})\to\cO_Y(-D_{\alpha})\to\cO_Y(-D_{\alpha})/\cO_Y(-D_{>\alpha})\to 0
\end{equation}
with $\Omega_Y^i({\rm log}\,E)$ and taking the long exact sequence for higher direct images, we see that
it suffices to show that 
\begin{equation}\label{eq_cor3_resolution}
R^j\pi_*\big(\Omega_Y^i({\rm log}\,E)\otimes (\cO_Y(-D_{\alpha})/\cO_Y(-D_{>\alpha}))\big)=0\quad\text{for}\quad i+j>n.
\end{equation}
Note that Corollary~\ref{cor1_resolution} and the fact that the de Rham complex of a Hodge module is concentrated in nonpositive degrees imply that
\begin{equation}\label{eq_van_cor3_resolution}
R^j\pi_*\big(\Omega_{Y/\A^1}^i({\rm log}\,E)\otimes (\cO_Y(-D_{\alpha})/\cO_Y(-D_{>\alpha}))\big)=0\quad\text{if}\quad i+j>n-1.
\end{equation}
By tensoring with $\cO_Y(-D_{\alpha})/\cO_Y(-D_{>\alpha})$ the short exact sequence (\ref{eq2_def_rel_diff}), with $\ell=i$, and taking the long
exact sequence for higher direct images, we obtain the vanishing in (\ref{eq_cor3_resolution}).

The assertion in ii) follows in the same way via (\ref{eq_van_cor3_resolution}) if we show that
$$R^j\pi_*\Omega_{Y/\A^1}^i({\rm log}\,E)(-E)=0\quad\text{if}\quad i+j>n-1.$$
This is clear if $j=0$ (since in this case $i>n-1$ and thus $\Omega_{Y/\A^1}^i({\rm log}\,E)=0$), hence we assume $j>0$.
If we denote the sheaf on the left-hand side by $\cF$, note first that since $j>0$, we have ${\rm Supp}(\cF)\subseteq Z$. 
On the other hand, the long exact sequence for higher direct images 
associated to the exact sequence obtained by tensoring (\ref{eq_SES1}) with $\Omega_{Y/\A^1}^i({\rm log}\,E)$
and the vanishing (\ref{eq_van_cor3_resolution}) give
$$R^j\pi_*\big(\Omega_{Y/\A^1}^i({\rm log}\,E)\otimes \cO_{Y/\A^1}(-D_{>\alpha})\big)\to R^j\pi_*\big(\Omega_Y^i({\rm log}\,E)\otimes \cO_Y(-D_{\alpha})\big)$$
$$\to R^{j+1}\pi_*\big(\Omega_{Y/\A^1}^i({\rm log}\,E)\otimes (\cO_Y(-D_{\alpha})/\cO_Y(-D_{>\alpha}))\big)=0$$
for all $\alpha\in (0,1]$. By successively applying this, we see that the canonical morphism
$$\cF\otimes_{\cO_X}\cO_X(-Z)=R^j\pi_*\big(\Omega_{Y/\A^1}^i({\rm log}\,E)(-E-D)\big)=R^j\pi_*\big(\Omega_Y^i({\rm log}\,E)\otimes \cO_Y(-D_{>1})\big)$$
$$\to
R^j\pi_*\big(\Omega_Y^i({\rm log}\,E)\otimes \cO(-D_{>0})\big)=R^j\pi_*\Omega_{Y/\A^1}^i({\rm log}\,E)(-E)=\cF$$
is surjective. Therefore, we have $\cF\otimes_{\cO_X}\cO_Z=0$, and Nakayama's Lemma implies that ${\rm Supp}(\cF)\cap Z=\emptyset$.
Since ${\rm Supp}(\cF)\subseteq Z$, we conclude that $\cF=0$.
\end{proof}

The following corollary will lead to the proof of our main result. For every $\alpha\in (0,1)\cap\Q$, we put
$${\mathcal E}_Y^{q, \alpha}=\cO_Y(-D_{\alpha})\otimes_{\cO_Y}\Omega_Y^q({\rm log}\,E)\quad\text{and}\quad{\mathcal E}_Y^{q, >\alpha}=\cO_Y(-D_{>\alpha})\otimes_{\cO_Y}\Omega_Y^q({\rm log}\,E)$$
and consider for every $i$ the canonical morphism
$$\gamma_i\colon {\mathbf R}\pi_*({\mathcal E}^{n-i,>\alpha}_Y)\to {\mathbf R}\pi_*({\mathcal E}^{n-i,\alpha}_Y)$$
induced by the inclusion $\cO_Y(-D_{>\alpha})\hookrightarrow \cO_Y(-D_{\alpha})$.

\begin{cor}\label{cor2_resolution}
With the above notation, if $\alpha\in (0,1)\cap\Q$ and $p\in\Z_{\geq 0}$ is such that $\widetilde{\alpha}(Z)\geq p$, 
then the following hold:
 \begin{enumerate}
 \item[i)] $\gamma_i$ is an isomorphism for $i\leq p-1$.
 \item[ii)] $\cH^q(\gamma_p)$ is an isomorphism for $q\neq p$ and $\cH^p(\gamma_p)$ is injective.
 \item[iii)] If $\widetilde{\alpha}(Z)\geq p+\alpha$, then we have $\widetilde{\alpha}(Z)> p+\alpha$ if and only if $\gamma_p$ is an isomorphism. 
 \end{enumerate}
\end{cor}

\begin{proof}
The assertions are local on $X$ and they hold trivially on $X\smallsetminus Z$, hence we may and will assume that $Z$ is defined
by $f\in\cO_X(X)$ and the morphism $X\smallsetminus Z\overset{f}\longrightarrow\A^1\smallsetminus \{0\}$ is smooth.
We first prove the corresponding result when the sheaves $\Omega_Y^j({\rm log}\,E)$ are replaced by $\Omega^j_{Y/\A^1}({\rm log}\,E)$.
More precisely, let
$${\mathcal E}_{Y/\A^1}^{q, \alpha}=\cO_Y(-D_{\alpha})\otimes_{\cO_Y}\Omega_{Y/\A^1}^q({\rm log}\,E), \,\,
{\mathcal E}_{Y/\A^1}^{q, >\alpha}=\cO_Y(-D_{>\alpha})\otimes_{\cO_Y}\Omega_{Y/\A^1}^q({\rm log}\,E),$$
and consider for every $i$ the canonical morphism 
$$\beta_i\colon {\mathbf R}\pi_*{\mathcal E}_{Y/\A^1}^{n-1-i,>\alpha}\to {\mathbf R}\pi_*{\mathcal E}_{Y/\A^1}^{n-1-i,\alpha}$$
induced by the inclusion $\cO_Y(-D_{>\alpha})\hookrightarrow \cO_Y(-D_{\alpha})$. 
With this notation, by combining Corollaries~\ref{cor2_char} and \ref{cor1_resolution}, and using the long exact sequence in cohomology associated to the distinguished triangle
$${\mathbf R}\pi_*\big(\cO_Y(-D_{>\alpha})
\otimes_{\cO_Y}\Omega_{Y/\A^1}^{n-1-i}({\rm log}\,E)\big)\to {\mathbf R}\pi_*\big(\cO_Y(-D_{\alpha})
\otimes_{\cO_Y}\Omega_{Y/\A^1}^{n-1-i}({\rm log}\,E)\big)$$ 
$$\to{\mathbf R}\pi_*\big((\cO_Y(-D_{\alpha})/\cO_Y(-D_{>\alpha}))\otimes_{\cO_Y}\Omega_{Y/\A^1}^{n-i-1}({\rm log}\,E)\big)
\overset{+1}\longrightarrow,$$
we deduce that the following hold:
\begin{enumerate}
\item[i')] $\beta_i$ is an isomorphism for all $i\leq p-1$.
\item[ii')] $\cH^q(\beta_p)$ is an isomorphism for all $q\neq p$ and $\cH^p(\beta_p)$ is injective.
\item[iii')] If $\widetilde{\alpha}(Z)\geq p+\alpha$, then we have $\widetilde{\alpha}(Z)>p+\alpha$ if and only if $\beta_p$ is an isomorphism. 
\end{enumerate}
Note that $\cH^q(\beta_p)$ is an isomorphism for $q>p$
since both the domain and the target vanish by Corollary~\ref{cor3_resolution}ii).

Note now that the exact sequence (\ref{eq2_def_rel_diff}) gives 
 a commutative diagram 
$$\xymatrix{
 {\mathbf R}\pi_*({\mathcal E}^{n-1-i,>\alpha}_{Y/\A^1})\ar[d]_{\beta_i}\ar[r]& {\mathbf R}\pi_*({\mathcal E}^{n-i,>\alpha}_Y)\ar[d]^{\gamma_i}\ar[r] &{\mathbf R}\pi_*({\mathcal E}^{n-i,>\alpha}_{Y/\A^1})
 \ar[d]_{\beta_{i-1}}\ar[r] &{\mathbf R}\pi_*({\mathcal E}^{n-1-i,>\alpha}_{Y/\A^1})[1]\ar[d]^{\beta_i[1]}\\
 {\mathbf R}\pi_*({\mathcal E}^{n-i-1,\alpha}_{Y/\A^1})\ar[r]& {\mathbf R}\pi_*({\mathcal E}^{n-i,\alpha}_Y)\ar[r] &{\mathbf R}\pi_*({\mathcal E}^{n-i,\alpha}_{Y/\A^1})\ar[r]&
 {\mathbf R}\pi_*({\mathcal E}^{n-i-1,\alpha}_{Y/\A^1})[1],}$$ 
 in which the rows are exact triangles. 
 The fact that $\gamma_i$ is an isomorphism for $i\leq p-1$ follows directly from i') above and the 5-Lemma. Let's take now $i=p$. The fact that $\cH^q(\gamma_p)$ is an isomorphism for $q>p$
 follows from the fact that both the source and the target vanish by Corollary~\ref{cor3_resolution}i). The fact that $\cH^q(\gamma_p)$ is an isomorphism for $q<p$ and injective for $q=p$
 follows from ii') above and the 5-Lemma. Finally, since $\cH^{p-1}(\beta_{p-1})$, $\cH^p(\beta_{p-1})$ and $\cH^{p+1}(\beta_p)$ are isomorphisms, 
 the 5-Lemma implies that $\cH^p(\gamma_p)$ is surjective (hence an isomorphism) if and only if $\cH^p(\beta_p)$ is surjective (hence an isomorphism) and by iii') above,
 this is the case if and only if $\widetilde{\alpha}(Z)>p+\alpha$.
\end{proof}

We can now give the proof of our main result.

\begin{proof}[Proof of Theorem~\ref{thm_main}]
Let $0=\alpha_0<\alpha_1<\ldots<\alpha_r=\alpha$ be such that $D_{\alpha}=D_{\alpha_i}$ for $1\leq i\leq r$ and $\alpha\in (\alpha_{i-1},\alpha_i]$, and thus $D_{>\alpha}=D_{\alpha_{i}}$
for $\alpha\in [\alpha_{i-1},\alpha_i)$. Applying Corollary~\ref{cor2_resolution}ii) for each of $\alpha_1,\ldots,\alpha_r$ gives the statement in i). 
Since each morphism 
$$R^q\pi_*\Omega_Y^{n-p}({\rm log}\,E)(-D_{>\alpha_i})\to R^q\pi_*\Omega_Y^{n-p}{(\rm log}\,E)(-D_{\alpha_i}),$$
with $1\leq i\leq r$, is injective, it follows that they are all isomorphisms if and only if the composition is an isomorphism. 
Therefore Corollary~\ref{cor2_resolution}iii) gives the statement in ii). This completes the proof of the theorem.
\end{proof}

\begin{eg}\label{eg_multiplier}
If $p=0$, then we automatically have $\widetilde{\alpha}(Z)>0$. In this case, the condition $\widetilde{\alpha}(Z)>\alpha$, for $\alpha\in (0,1)$, is equivalent to ${\rm lct}(Z)>\alpha$. 
Since the morphism (\ref{eq_thm_main}) gets identified to the inclusion
$$\cJ(\alpha Z)\otimes_{\cO_X}\omega_X\hookrightarrow\omega_X,$$
where $\cJ(\alpha  Z)$ is the multiplier ideal of $\alpha Z$,
we recover the well-known condition that ${\rm lct}(Z)>\alpha$ if and only if $\cJ(\alpha Z)=\cO_X$ (see \cite[Chapter~9.3.B]{Lazarsfeld}).
\end{eg}

We deduce the dual version of Theorem~\ref{thm_main}:

\begin{proof}[Proof of Corollary~\ref{dual_formulation}]
Since 
$$\Omega_Y^{n-p}({\rm log}\,E)^{\vee}\simeq \Omega_Y^p({\rm log}\,E)\otimes_{\cO_Y}\omega_Y(E)^{-1},$$
Grothendieck duality for the morphism $\pi$ gives canonical isomorphisms
$${\mathbf R}{\mathcal Hom}_{\cO_X}\big({\mathbf R}\pi_*\Omega_Y^{n-p}({\rm log}\,E)(-E),\omega_X\big)\simeq {\mathbf R}\pi_*{\mathbf R}{\mathcal Hom}_{\cO_Y}\big(\Omega_Y^{n-p}({\rm log}\,E)(-E),\omega_Y\big)$$
$$\simeq {\mathbf R}\pi_*\Omega_Y^p({\rm log}\,E)\quad\text{and}$$
$${\mathbf R}{\mathcal Hom}_{\cO_X}\big({\mathbf R}\pi_*\Omega_Y^{n-p}({\rm log}\,E)(-E-\lfloor\alpha D\rfloor),\omega_X\big)\simeq {\mathbf R}\pi_*{\mathbf R}{\mathcal Hom}_{\cO_Y}\big(\Omega_Y^{n-p}({\rm log}\,E)(-E
-\lfloor\alpha D\rfloor),\omega_Y\big)$$
$$\simeq {\mathbf R}\pi_*\Omega_Y^p({\rm log}\,E)\big(\lfloor\alpha D\rfloor\big).$$
Since a morphism $u$ in $D^b_{\rm coh}(X)$ is an isomorphism if and only if ${\mathbf R}{\mathcal Hom}_{\cO_X}(u,\omega_X)$ is an isomorphism, it follows from Theorem~\ref{thm_main} that
if $\widetilde{\alpha}(Z)>p$, then we have $\widetilde{\alpha}(Z)>p+\alpha$ if and only if the canonical morphism 
$${\mathbf R}\pi_*\Omega_Y^p({\rm log}\,E)\to {\mathbf R}\pi_*\Omega_Y^p({\rm log}\,E)\big(\lfloor\alpha D\rfloor\big)$$
is an isomorphism. Of course, this is the case if and only if we have isomorphisms when taking the cohomology on both sides.
However, since we assume $\widetilde{\alpha}(Z)>p$, we have
$$R^q\pi_*\Omega_Y^p({\rm log}\,E)=0\quad\text{for all}\quad q\geq 1$$
by Theorem~\ref{thm_char}, while the morphism 
$$\pi_*\Omega_Y^p({\rm log}\,E)\to\pi_*\Omega_Y^p({\rm log}\,E)\big(\lfloor\alpha D\rfloor\big)$$
is an isomorphism by (\ref{eq1_reflexive}).
We thus obtain the assertion in the corollary.
\end{proof}

\begin{rmk}\label{rmk_DuBois}
Let's apply Corollary~\ref{dual_formulation} with $\alpha=1-\epsilon$, where $0<\epsilon\ll 1$, so $\lfloor\alpha D\rfloor=D-E$. The assertion in the corollary then says, via the projection formula, that if $\widetilde{\alpha}(Z)>p$,
then we have $\widetilde{\alpha}(Z)\geq p+1$ if and only if
\begin{equation}\label{eq_condition}
R^q\pi_*\Omega_Y^p({\rm log}\,E)(-E)=0\quad\text{for all}\quad q\geq 1.
\end{equation}
Recall now that by \cite[Proposition~3.3]{Steenbrink85}, we have a canonical exact triangle
$${\mathbf R}\pi_*\Omega_Y^p({\rm log}\,E)(-E)\longrightarrow\Omega_X^p\to\underline{\Omega}_Z^p\overset{+1}\longrightarrow,$$
where $\underline{\Omega}_Z^p$ is the $p$th Du Bois complex of $Z$ and the second map in the triangle is the composition of
canonical maps
$$\Omega_X^p\longrightarrow \Omega_Z^p\overset{\gamma_p}\longrightarrow\underline{\Omega}_Z^p.$$
We now show that condition (\ref{eq_condition}) above is satisfied if and only if $\gamma_p$ is an isomorphism.

The exact triangle implies that condition (\ref{eq_condition})  is equivalent to the fact that $\cH^i(\underline{\Omega}_Z^p)=0$ for all $i\geq 1$ and the canonical map
$\Omega_X^p\to\cH^0(\underline{\Omega}_Z^p)$ is surjective (equivalently, $\cH^0(\gamma_p)$ is surjective). If $p=0$, then $\cH^0(\gamma_0)$ being surjective
clearly implies that it is bijective.

Suppose now that $p\geq 1$. Since we assume that $\widetilde{\alpha}(Z)>p$, we first see that $Z$ has rational singularities, and thus $\cH^0(\underline{\Omega}_Z^p)=\Omega_Z^{[p]}$,
the sheaf of reflexive $p$-differentials, by \cite[Corollary~1.11]{KS}. 
On the other hand, since $\widetilde{\alpha}(Z)>p$, we have $r:={\rm codim}_Z(Z_{\rm sing})\geq 2p\geq p+1$. By
\cite[Theorem~1.11]{Graf}, this implies that the kernel of the canonical map $\Omega_Z^p\to\Omega_Z^{[p]}$ is $0$.
We thus conclude that (\ref{eq_condition}) holds if and only if $\gamma_p$ is an isomorphism.
We have thus shown that if $\widetilde{\alpha}(Z)>p$, then we have $\widetilde{\alpha}(Z)\geq p+1$ if and only if $\gamma_p$ is an isomorphism. 

In fact, it is known that for every reduced hypersurface $Z$, we have $\widetilde{\alpha}(Z)\geq p+1$ if and only if $\gamma_i$ is an isomorphism for all $i\leq p$: see 
\cite{MOPW} and \cite{Saito_et_al} for the two implications. It is interesting to note that while we have obtained another proof for the ``only if" part, the converse does not follow
from our argument since the assumption $\widetilde{\alpha}(Z)>p$ does not allow us to do induction on $p$. 
\end{rmk}

\section{An application to families with simultaneous log resolutions}

We conclude the paper with an application of our main results to the constancy of the minimal exponent in a proper family of hypersurfaces that admit a simultaneous log resolution. This answers a question of Radu Laza.

The setup in this section is the following. Suppose that $\mu\colon X\to T$ is a smooth proper morphism of smooth, irreducible complex algebraic varieties. Let $Z\hookrightarrow X$ be a relative effective Cartier divisor over $T$, that is,
$Z$ is an effective Cartier divisor on $X$ such that for every $t\in T$, its fiber $Z_t$ is an effective Cartier divisor in the smooth variety $X_t=\mu^{-1}(t)$. We assume that we have a proper morphism 
$\pi\colon Y\to X$ that is an isomorphism over $X\smallsetminus Z$ and such that $\mu\circ\pi$ is smooth and a simple normal crossing divisor $E=\sum_{i=1}^NE_i$ on $Y$ such that the following hold:
\begin{enumerate}
\item[i)] $E$ also has relative simple normal crossings\footnote{This means that for every subset $J\subseteq\{1,\ldots,N\}$, if $E_J=\cap_{i\in J}E_i$, then $E_J$ is smooth over $T$.}.
\item[ii)] We have $\pi^*(Z)=\sum_{i=1}^Na_iE_i$ for some positive integers $a_1,\ldots,a_N$.
\end{enumerate}

Under the above hypothesis, we get the constancy of the minimal exponent:

\begin{thm}\label{thm_last}
With the above notation, we have $\widetilde{\alpha}(X_t,Z_t)=\widetilde{\alpha}(X,Z)$ for all $t\in T$. 
\end{thm}

Before giving the proof of the theorem, we prove the following lemma, which is where we use the results in this paper.
\begin{lem}\label{lem_thm_last}
With the notation in Theorem~\ref{thm_last}, suppose in addition that $T$ is a curve and that $Z$ and all $Z_t$ are reduced. 
If ${t_0}\in T$ is a point such that $\widetilde{\alpha}(X_{t_0},Z_{t_0})\geq \gamma$ for some $\gamma\in \Q_{>0}$ and $\widetilde\alpha(X_t,Z_t)>\gamma$ for any other $t$ in some neighborhood of ${t_0}$, then we have $\widetilde\alpha(X_{t_0},Z_{t_0})> \gamma$.
\end{lem}

\begin{proof}
Without any loss of generality, we may assume that $\widetilde{\alpha}(X_t,Z_t)>\gamma$ for all $t\in T\smallsetminus\{t_0\}$. 
Using \cite{MPV}*{Theorem E(1)}, we deduce that $\widetilde{\alpha}(X\smallsetminus X_{t_0},Z\smallsetminus Z_{t_0})>\gamma$, and since
$\widetilde{\alpha}(X_{t_0},Z_{t_0})\geq\gamma$, by
Inversion of Adjunction for the minimal exponent~\cite{C24}*{Theorem 1.1}, we get $\widetilde{\alpha}(X,Z)>\gamma$. Let us write $\gamma=p+\alpha$, for some integer $p \geq 0$ and $\alpha\in [0,1)\cap \Q$. We put $D=\pi^*(Z)$. Let us denote by $\nu\colon Y\to T$ the composition $\mu\circ\pi$. As before, for $t\in T$, we write
 $Y_t=\nu^{-1}(t)$, $E_t=E\cap Y_t$, and $D_t= D\cap Y_t$.	We first show that 
	\[
	R^q \pi_* \Omega_{Y_{t_0}}^i(\log E_{t_0})\big(\lfloor{\alpha D_{t_0}} \rfloor\big)=0\quad\text{for all}\quad q>0,\,i\leq p.
	\] 
	We may and will assume that we have $z\in\cO(T)$ that gives an algebraic coordinate on $T$ and such that $z(t_0)=0$. 
	Let  $\Omega_{ Y/T}(\log\,E) $ be the cokernel of the injective morphism $\cO_Y\to \Omega_{  Y}(\log\,E)$ that maps $1$ to $du$. Since $E$ has relative simple normal crossings over $T$, the sheaf $\Omega_{Y/T}(\log\,E)$ is locally free. Put $\Omega^i_{Y/T}(\log\,E)=\wedge^i \Omega_{Y/T}(\log\,E)$ for every $i\geq 0$. Note that for every $i\geq 0$, we have a short exact sequence of locally free $\cO_Y$-modules
	\[
	0 \to  \Omega_{  Y/T}^{i-1}(\log   E)\big(\lfloor{\alpha   D} \rfloor\big) \xrightarrow{dz \wedge-} \Omega_{  Y}^i(\log   E)\big(\lfloor{\alpha   D} \rfloor\big) \to \Omega_{Y/T}^i(\log\,E)\big(\lfloor{\alpha   D} \rfloor\big) \to 0,
	\]
	which induces the exact triangle
	\begin{equation}\label{eq_exacttrig}
		\bR\pi_* \Omega_{  Y/T}^{i-1}(\log\,E)\big(\lfloor{\alpha   D} \rfloor\big) \to \bR\pi_*\Omega_{Y}^i(\log\,E)\big(\lfloor{\alpha   D} \rfloor\big) \to \bR \pi_* \Omega_{Y/T}^i(\log\,E)\big(\lfloor{\alpha   D} \rfloor\big) \xrightarrow{+1}.
	\end{equation}
	Since $\widetilde\alpha(X,Z)>p+\alpha$, it follows from Theorem~\ref{thm_char} and Corollary~\ref{dual_formulation} that 
	$$R^q\pi_*\Omega_{Y}^i(\log\,E)\big(\lfloor{\alpha   D} \rfloor\big)=0\quad\text{for all}\quad q>0,\,i\leq p.$$ 
	Taking the long exact sequence associated to~\eqref{eq_exacttrig} and using
	induction on $i$, we obtain
	\begin{equation}\label{eq_vanishing1}
	R^q \pi_*\Omega_{  Y/T}^i(\log\,E)\big(\lfloor{\alpha   D} \rfloor\big)=0	 \quad \text{for all}\quad q>0,\, i\leq p.
	\end{equation}
	Note that the resolution $\cO_T \overset{z}\longrightarrow \cO_T$ of $\cO_{\{t_0\}}$ implies that in the derived category of coherent sheaves on $X$ (respectively $Y$), we can represent
	$\cO_{X_{t_0}}$ (respectively $\cO_{Y_{t_0}}$) by the the complex $\cO_X\overset{z}\longrightarrow \cO_X$ (respectively, by $\cO_Y\overset{z}\longrightarrow\cO_Y$). In particular, we have
	${\mathbf L}\pi^*\cO_{X_{t_0}}\simeq\cO_{Y_{t_0}}$.
	 Using the projection formula, we get
	\[
	\begin{aligned}
	\pi_* \Omega_{Y/T}^i(\log\,E)\big(\lfloor{\alpha   D} \rfloor\big) \otimes^{\mathbf L}_{\cO_{X}} {\cO_{X_{t_0}}} &=	\bR \pi_* \Omega_{Y/T}^i(\log\,E)\big(\lfloor{\alpha   D} \rfloor\big) \otimes^{\mathbf L}_{\cO_{X}} {\cO_{  X_{t_0}}} \quad \text{by}~\eqref{eq_vanishing1}\\
	&= \bR \pi_*\left(\Omega_{Y/T}^i(\log\,E)\big(\lfloor{\alpha   D}\rfloor\big)\otimes^{\mathbf L}_{\cO_{Y}} {\mathbf L}\pi^*{\cO_{X_{t_0}}}\right) \\
	&= \bR \pi_*\left(\Omega_{Y/T}^i(\log\,E)\big(\lfloor{\alpha   D}\rfloor\big)\otimes_{\cO_{  Y}} {\cO_{  Y_{t_0}}}\right) \\
	&= \bR \pi_*\Omega_{  Y_{t_0}}^i(\log\,E_{t_0})\big(\lfloor{\alpha   D_{t_0}}\rfloor\big).
	\end{aligned}
	\]
	We thus conclude that 
	\begin{equation}\label{eq_vanishing2}
		R^q \pi_* \Omega_{Y_{t_0}}^i(\log\,E_{t_0})\big(\lfloor{\alpha D_{t_0}} \rfloor\big)=0 \quad \text{for all}\quad q>0,\,  i\leq p.
	\end{equation}
	 
	If $\alpha>0$, then Corollary~\ref{dual_formulation} gives $\widetilde\alpha(X_{t_0},Z_{t_0})>p+\alpha$, hence the proof of the lemma is complete in this case.
	 The remaining case is when $\alpha=0$ and $p\geq 1$. If $p\geq 2$, then the condition $\widetilde\alpha(X_{t_0},Z_{t_0})\geq p$ implies $\codim_{Z_{t_0}}\big(({Z_{t_0}})_{\rm sing}\big)\geq 2p-1 \geq 3$ by, for example, ~\cite[Lemma~2.2]{MOPW}. According to Theorem~\ref{thm_char} and Remark~\ref{codim_sing_locus}, we deduce using also (\ref{eq_vanishing2}), that $\widetilde\alpha(X_{t_0},Z_{t_0})>p$.  Finally, if $p=1$, then $Z$ has rational singularities by~\cite[Theorem~0.4]{Saito-B} as $\widetilde\alpha(X,Z)>1$. 
	 Let $Z'$ be an irreducible component of $Z$ (note that since $Z$ is normal, its irreducible components are pairwise disjoint). If $E'$ is the strict transform of $Z'$ on $Y$, 
	 since $Z$ has rational singularities, we have
	$\cO_{Z'} \simeq \bR\pi_* \cO_{E'}$. Hence, since $Z'$ is flat over $T$, by the projection formula, we have 
	\[
	\cO_{Z'_{t_0}}\simeq \bR\pi_* \cO_{E'}\otimes_{\cO_{Z'}}^{\bf L} \cO_{Z'_{t_0}}\simeq \bR\pi_* \cO_{E'_{t_0}}.
	\] 
	By our assumption, $E'_{t_0}$ is a resolution of singularities of $Z'_{t_0}$. Hence, $Z'_{t_0}$ has rational singularities as well,
	and thus $Z_{t_0}$ has the same property. Another application of \cite[Theorem~0.4]{Saito-B} gives $\widetilde\alpha(X_{t_0},Z_{t_0})>1$.
	This completes the proof of the lemma.
\end{proof}

We can now prove the main result of this section.

\begin{proof}[Proof of Theorem~\ref{thm_last}]
Since $K_{Y_t/X_t}=K_{Y/X}\vert_{Y_t}$ for every $t\in T$, it is straightforward to see, using the description of the log canonical threshold in (\ref{formula_lct}) that 
$\lct(X_t,Z_t)=\lct(X,Z)$ for all $t\in T$.
The assertion in the theorem is thus clear when this common value is $<1$, hence from now on we may and will assume that $Z$ and all $Z_t$ are reduced. 
	
Let us recall first some basic facts about the behavior of minimal exponents in families. There is a nonempty open subset 	$U\subseteq T$
such that $\widetilde{\alpha}_P(X,Z)=\widetilde{\alpha}_P(X_t,Z_t)$ for all $t\in U$ and all $P\in Z_t$ (for the definition of the local version of the minimal exponent, see
Remark~\ref{rmk_local_min_exponent}). This is a consequence of the description of the minimal exponent in terms of Hodge ideals 
(see \cite[Lemma~5.3]{MPV}) and of the behavior of Hodge ideals under restriction to general hypersurfaces (see \cite[Theorem~13.1]{MPQ}).

Moreover, if $U\subseteq T$ is as above, then after possibly replacing $U$ by a smaller open subset, we may and will assume that 
$\widetilde{\alpha}(X_t,Z_t)=\alpha$ for all $t\in U$. 
Indeed, for every $\alpha>0$, the set 
$$W_{\alpha}:=\big\{P\in Z\cap \mu^{-1}(U)\mid \widetilde{\alpha}_P(X,Z)\leq\alpha\big\}$$
is a closed subset of $\mu^{-1}(U)$ and we have only finitely many distinct such subsets. If $\alpha$ is minimal with the property that $W_{\alpha}$ dominates $T$, it follows that
for every $t\in U\smallsetminus\bigcup_{\beta<\alpha}\mu(W_{\beta})$, we have $\widetilde{\alpha}(X_t,Z_t)=\alpha$. 

The first step is to show that $\widetilde{\alpha}(X_t,Z_t)$ is independent of $t\in T$. We will then deduce that this constant value is $\widetilde{\alpha}(X,Z)$. 
In order to prove independence of $t$, we may and will assume that $T$ is an affine variety. Moreover, since any two points on $T$ lie on 
the image of
some smooth curve $C$, after taking the fiber product with $C\to T$, we may and will assume that $T$ is a curve. 
		
In this case, it follows from Lemma~\ref{lem_thm_last}
 that $\widetilde{\alpha}(X_s,Z_s)\geq\alpha$ for $s\in T\smallsetminus U$. However, a version of the semicontinuity of minimal exponents
(see \cite[Theorem~E(2)]{MPV}) implies now that $\widetilde{\alpha}(X_s,Z_s)\leq\alpha$ for all $s\in T\smallsetminus U$. Indeed, let us fix $s\in T\smallsetminus U$. It is a consequence of the behavior of the minimal exponent under restriction to a hypersurface
(see \cite[Theorem~E(1)]{MPV})
 that 
 $$\widetilde{\alpha}(X_s,Z_s)\leq \widetilde{\alpha}_P(X_s,Z_s)\leq \widetilde{\alpha}_P(X,Z)$$
 for all $P\in Z_s$. Since $\mu$ is proper, it follows that there is an open neighborhood $V$ of $s$ such that 
$$\widetilde{\alpha}(X_s,Z_s)\leq\widetilde{\alpha}_P(X,Z)\quad\text{for all}\quad t\in V,\,P\in Z_t.$$ 
By taking $t\in V\cap U$, we see that $\widetilde{\alpha}(X_s,Z_s)\leq\alpha$.
We thus conclude that $\widetilde{\alpha}(X_t,Z_t)=\alpha$ for all $t\in T$.

Suppose now that $T$ is arbitrary (not necessarily one-dimensional). We have seen 
that $\widetilde{\alpha}(X_t,Z_t)=\gamma$ for all $t\in T$ and some $\gamma$ and we want to show that $\widetilde{\alpha}(X,Z)=\gamma$.
Recall that we have a nonempty open subset $U$ of $T$ such that $\widetilde{\alpha}_P(X_t,Z_t)=\widetilde{\alpha}_P(X,Z)$ for all $t\in U$ and all $P\in Z_t$. This implies that $\widetilde{\alpha}(X,Z)\leq\gamma$. 
If this is not an equality, then there is $P\in Z$  such that $\widetilde{\alpha}_P(X,Z)<\gamma$. However, by the behavior of the minimal exponent under restriction to smooth subvarieties (which 
is a consequence of \cite[Theorem~E(1)]{MPV}), we conclude that if $s=\mu(P)$, then
$$\gamma=\widetilde{\alpha}(X_s,Z_s)\leq \widetilde{\alpha}_P(X_s,Z_s)\leq \widetilde{\alpha}_P(X,Z)<\gamma,$$
a contradiction. This completes the proof of the theorem.	
\end{proof}

\begin{rmk}
The properness assumption in Theorem~\ref{thm_last} is important: if we drop it, then the conclusion fails for rather trivial reasons. For example, if $\mu\colon X\to T$ and $Z\hookrightarrow X$
are as in the statement of the theorem and $Z\to T$ is not smooth, then for $t_0\in T$ and $\mu'\colon X'=X\smallsetminus (Z_{t_0})_{\rm sing}\to X$ and 
$Z'=Z\smallsetminus (Z_{t_0})_{\rm sing}$,
the existence of a simultaneous resolution still holds, but $\widetilde{\alpha}(Z'_{t_0})=\infty>\widetilde{\alpha}(Z_t)$ for $t\neq t_0$.

Furthermore, this can't be fixed\footnote{It is typical, when considering the semicontinuity of singularity invariants in non-proper families, to choose a section
and only consider the local invariants at the points in the image of the section (see, for example, \cite[Theorem~E(2)]{MPV}).}
 by considering the local minimal exponent at the points in the image of a section. In fact, the assertion in Theorem~\ref{thm_last} can fail \emph{even for proper families} 
 for the local minimal exponent. More precisely, with the assumptions in the theorem,
it is not true that if
$s\colon T\to X$ is a section of $\pi$, then $\widetilde{\alpha}_{s(t)}(X_t,Z_t)$ is independent of $T$. Indeed, suppose that 
we consider a reduced singular hypersurface $H$ in a smooth variety $W$. Let $g\colon \widetilde{H}\to H$ be a resolution of singularities of $H$ and $h\colon \widetilde{W}\to W$
be a log resolution of $(W,H)$ that is an isomorphism over $W\smallsetminus H$. Let's consider $\mu\colon X=W\times \widetilde{H}\to \widetilde{H}=T$
given by the second projection and $s\colon T\to X$ the section given by $(i\circ g,\id)$, where $i\colon H\hookrightarrow W$ is the inclusion. 
If $Z=H\times\widetilde{H}\hookrightarrow X$ is given by $(i,\id)$ and we consider the morphism $h\times\id\colon Y=\widetilde{W}\times \widetilde{H}\to W\times \widetilde{H}=X$,
then we see that the hypothesis in Theorem~\ref{thm_last} is satisfied. On the other hand, for every $t\in \widetilde{W}$, we have
$$\widetilde{\alpha}_{s(t)}(X_t,Z_t)=\widetilde{\alpha}_{g(t)}(W,H)$$
and this is not independent of $t$. 
\end{rmk}

\end{document}